\documentclass[twoside,11pt]{amsart}

\usepackage{amsmath,mathtools}
\usepackage{xspace}
\usepackage{amssymb}
\usepackage{url}
\usepackage{pifont}
\usepackage{enumerate}
\usepackage{tikz}
\usepackage{lipsum}
\usepackage[matrix,arrow]{xy}

\usepackage{bbm}
\usepackage{verbatim}
\usepackage{graphicx,color}
\usepackage[hyperindex=true,plainpages=false,colorlinks=false,pdfpagelabels]{hyperref}

\DeclareFontFamily{U}{mathx}{\hyphenchar\font45}
\DeclareFontShape{U}{mathx}{m}{n}{
      <5> <6> <7> <8> <9> <10>
      <10.95> <12> <14.4> <17.28> <20.74> <24.88>
      mathx10
      }{}
\DeclareSymbolFont{mathx}{U}{mathx}{m}{n}
\DeclareFontSubstitution{U}{mathx}{m}{n}
\DeclareMathAccent{\widecheck}{0}{mathx}{"71}
\DeclareMathAccent{\widetilde}{0}{mathx}{"72}
\DeclareMathAccent{\widebar}{0}{mathx}{"73}
\DeclareMathAccent{\widevec}{0}{mathx}{"74}
\DeclareMathAccent{\widehat}{0}{mathx}{"70}
\DeclareMathAccent{\widefrown}{0}{mathx}{"75}
\DeclareMathAccent{\chinesehat}{0}{mathx}{"69}
\theoremstyle{plain}
\newtheorem{The}{Theorem}
\newtheorem*{The*}{Theorem}
\newtheorem{Pro}{Proposition}
\newtheorem{Lem}{Lemma}
\newtheorem{Cor}{Corollary}
\newtheorem*{Cor*}{Corollary}
\newtheorem{Con}{Conjecture}

\theoremstyle{definition}

\newtheorem{Rem}{Remark}
\newtheorem{Exa}{Example}
\newtheorem*{Rem*}{Remark}
\newtheorem*{Con*}{Convention}
\numberwithin{equation}{section}

\def\Res{{\,\rm Res}}
\DeclareMathOperator{\tr}{tr}

\DeclareMathOperator{\Id}{Id}

\DeclareMathOperator{\vol}{vol}

\renewcommand{\Im}{\operatorname{Im}}
\renewcommand{\Re}{\operatorname{Re}}

\newcommand{\wt}[1]{\widetilde{#1}}

\newcommand{\off}{\text{off}}
\newcommand{\diag}{\text{diag}}
\newcommand{\Lawson}{\mathfrak L}
\newcommand{\newPhi}{\vartheta}

\newcommand{\R}{\mathbb{R}}

\newcommand{\C}{\mathbb{C}}

\newcommand{\Z}{\mathbb{Z}}

\renewcommand{\S}{\mathbb{S}}  
\newcommand{\T}{\mathbb{T}}

\begin{document}
\title[Enclosed volume for CMC surfaces]{The Enclosed Volume for Periodic Constant Mean Curvature Surfaces}

\author{Lynn Heller}
\author{Sebastian Heller}
\author{Martin Traizet}

\bigskip
\noindent
\address{Beijing Institute of Mathematical Sciences and Applications\\
  Beijing 101408, China
}\\
\email{lynn@bimsa.cn}
 \noindent 
 \address{Beijing Institute of Mathematical Sciences and Applications\\
  Beijing 101408, China
}\\
\email{sheller@bimsa.cn}
 \noindent \address{Institut Denis Poisson, CNRS UMR 7350 \\
Facult\'e des Sciences et Techniques,
Universit\'e de Tours}\\
\email{martin.traizet@lmpt.univ-tours.fr}





\begin{abstract}
We establish a general formula for the enclosed volume of constant mean curvature (CMC) surfaces in Euclidean three space with translational periods forming a lattice.  The formula relates the volume to the surface area, a Wess-Zumino-Witten-type  term, and a newly defined curvature term of the associated family of flat connections, thereby extending the classical Minkowski formula for closed CMC surfaces. Interpreting the volume as a gauge-invariant quantity, we apply the result to a variety of examples and provide explicit computations. As an application, we construct a counterexample to the isoperimetric problem in $\T^2 \times \mathbb{R}$, disproving the conjecture that minimizers are restricted to spheres, cylinders, or pairs of planes.
\end{abstract}

\maketitle
\setcounter{tocdepth}{1}
\tableofcontents

\section{Introduction}

Constant mean curvature (CMC) surfaces are critical points of the area functional
under a volume constraint and play a central role in differential geometry.
For a closed CMC immersion
$f \colon \Sigma \to \mathbb{R}^3$
with mean curvature $H$, the classical Minkowski formula
\begin{equation}\label{eq:Minkowski-classical}
A(f) = 3 H V(f)
\end{equation}
relates surface area and enclosed volume. In the presence of translational periods, i.e., $f \colon \Sigma \to \mathbb{R}^3/\Gamma$ for a lattice $\Gamma \subset \R^3,$ the notion of enclosed volume remains
meaningful but requires a global formulation compatible with periodicity.
The purpose of this paper is to show that for periodic CMC surfaces the enclosed
volume admits a natural gauge-theoretic description in terms of the associated
family of flat connections.
Within this framework, the classical Minkowski formula is replaced by an identity
involving a curvature invariant of the associated family of flat connections.

A natural framework for such a formulation is provided by the gauge-theoretic
description of CMC surfaces (see \cite{Hi}).
Any CMC (periodic) immersion into $\mathbb{R}^3 \cong \mathfrak{su}(2)$ is encoded by its
associated family of flat $\mathrm{SL}(2,\mathbb{C})$-connections $\nabla^\lambda$,
unitary on the unit circle and trivial at the so-called Sym point $\lambda = 1$,
from which the immersion can be recovered via the Sym--Bobenko formula \cite{Bo}.
Many geometric quantities of the surface, such as the induced metric, the Hopf
differential, and the area, admit gauge-invariant expressions in terms of this
family. See \cite[Chapter 13]{Hi} for the case of tori, \cite[Theorem 8]{He13} for the case of symmetric genus 2 surfaces, and \cite{HHT1,HHT2} for the general case of compact CMC and minimal surfaces in $\mathbb S^3.$
Unlike the area, the enclosed volume is not determined by first-order
gauge-theoretic data.
Instead, it is governed by higher-order information of the associated family
$\nabla^\lambda$.

More precisely, let $\nabla^{e^{i\tau}}$ denote the associated family of a periodic
CMC surface.
Using a harmonic gauge, we define a scalar invariant
\[
K = \frac{i}{2\pi} \int_\Sigma \mathrm{tr}(D' \wedge D''),
\]
where $D'$ and $D''$ are the first and second derivatives at $\tau=0$ of the
gauged family.
This quantity is gauge invariant and depends only on the gauge class of the
$\nabla^\lambda$.
Geometrically, $K$ may be interpreted as a curvature term associated to the
second-order behaviour of the family in the moduli space of flat connections.

Throughout the paper we normalize the mean curvature to $H=1$.
Our first main result establishes a generalized Minkowski formula for periodic
CMC surfaces.
Let $f \colon \Sigma \to \mathbb{R}^3/\Gamma$ be a compact CMC surface and let
$\rho \colon H_1(\Sigma,\mathbb{Z}) \to \Gamma$ denote the period map.
Then
\begin{equation}\label{eq:Minkowski-generalized}
K
= -\frac{i}{2\pi} A(f)
+ \frac{3i}{2\pi} V(f)
- \frac{3i}{2\pi} V_\Sigma(\rho),
\end{equation}
where $V(f)$ is the enclosed (algebraic) volume of $f$ and
$V_\Sigma(\rho)$ is a  Wess--Zumino--Witten--type (WZW) term depending only on the harmonic
part of the period map.
When the translational lattice $\Gamma$ has rank at most two, the WZW term
vanishes identically, and the formula simplifies to
\[
K = -\frac{i}{2\pi} A(f) + \frac{3i}{2\pi} V(f).
\]
For closed surfaces, $K$ vanishes and the classical Minkowski formula is recovered.

An important consequence of this result is that the enclosed volume of a periodic
CMC surface depends only on the gauge class of its associated family of flat
connections.
This makes it possible to compute volumes purely from integrable-systems data.
We develop two complementary approaches to such computations: one based on
spectral data for equivariant surfaces, and a general residue formula expressing
$K$ in terms of the first two coefficients of a DPW potential.

In order to apply the formula, we construct Lawson-type families of doubly periodic CMC surfaces
in $\mathbb T^2 \times \mathbb{R}$.
Using the curvature formula together with explicit DPW data, we compute the
enclosed volume of these surfaces and compare their isoperimetric profiles with
those of the conjectured minimizers.
This yields a counterexample to the conjecture of Hauswirth-P\'erez-Romon-Ros \cite{HPRR}, showing that in the case of the equilateral torus $T^2$
the isoperimetric minimizers in $\mathbb T^2 \times \mathbb{R}$ are not restricted to
spheres, cylinders, or pairs of planes.

The present project grew out of our earlier work on computing area formulas for Lawson minimal surfaces in the three-sphere \cite{HHT2}. Following this, Antonio Ros asked whether similar gauge-theoretic methods could be applied to the corresponding constant mean curvature Lawson cousins in Euclidean three space, and in particular whether the enclosed volume could also be computed as a conserved quantity, with a view toward applications to the isoperimetric problem. Hitchin \cite{Hi2} showed that for minimal surfaces in the three-sphere the enclosed volume admits a gauge-theoretic interpretation as the holonomy of the Witten connection on the Chern-Simons bundle along a certain closed curve in the moduli space of flat connections. Motivated by this result, we first extended Hitchin's approach to constant mean curvature surfaces in $\S^3$ \cite{HHTS3}.  

In the Euclidean setting, the situation is qualitatively different. Periodic constant mean curvature surfaces in $\mathbb{R}^3$ do not close up, and the associated family of flat connections therefore does not define a closed loop in the moduli space of flat connections. As a result, the enclosed volume can no longer be recovered purely as a holonomy of the Chern-Simons bundle. Instead, new ingredients enter the picture: instead of the holonomy-type contribution, one must account for the curvature $K$, reflecting the non-closure of the family in moduli space. The Euclidean construction can thus be viewed as a refinement of the $\S^3$ theory, rather than a generalization, and this refined framework is developed in the present paper.

The paper is organized as follows.
Section~\ref{sec:r3} develops the general framework for describing CMC surfaces
in $\mathbb{R}^3$ with translational periods via families of flat connections and
introduces the definition of enclosed volume in this setting.
In Section~\ref{sec:Kenclosed} we derive the main formula relating the enclosed
volume, the area, the Wess--Zumino--Witten term, and the curvature $K$ of the
associated family at the Sym point.
Section~\ref{section:computeK} is devoted to the computation of the curvature $K$:
we first treat CMC tori using spectral data and then consider higher-genus CMC
surfaces with rank~$2$ period lattices, where $K$ is expressed in terms of residues
of their meromorphic DPW potentials.
In Section~\ref{section:lawson} we construct Lawson-type CMC surfaces in
$\mathbb{T}^2 \times \mathbb{R}$ and compute their enclosed volumes using the
preceding theory.
In Section~\ref{sec:iso} we obtain a counterexample to the conjectured
isoperimetric minimizers.
Finally, Appendix~\ref{section:IFT} contains the technical details for the
construction of DPW potentials for Lawson-type CMC surfaces.

\section{The enclosed volume for periodic CMC surfaces}\label{sec:r3}
This paper is motivated by the isoperimetric problem of flat three-manifolds of the form $\R^3/\Gamma,$ 
where $\Gamma$ is a discrete torsion-free subgroup of the group of Euclidean transformations.
Throughout the paper, assume that $\Gamma<(\R^3,+)$ is a lattice of translations. 
In particular, depending on the rank $r\in\{0,1,2,3\}$ of $\Gamma$, the quotient $\R^3/\Gamma$ is isometric to $\R^3$ (for $r=0$), $\S^1\times \R^2$ (for $r=1$), $T^2\times \R$ (for $r=2$), or $T^3$ (for $r=3$).

We consider surfaces $f\colon\Sigma\to\R^3/\Gamma.$ 
and refer to such maps $f$ as {\em periodic surfaces} -- that is, surfaces with {\em translational periods} corresponding to the lattice $\Gamma$.
 Every periodic surface $f$ gives rise to a well-defined $\R^3$-valued differential $df\in\Omega^1(\Sigma,\R^3)$ whose periods lie in $\Gamma.$
 Conversely, given a 1-form $\omega \in \Omega^1(\Sigma,\R^3)$ whose periods lie in $\Gamma$, one obtains a well-defined map $f\colon \Sigma \to \R^3/\Gamma$ such that $df=\omega$.

In this section, we first describe the integrable systems approach to periodic CMC surfaces via families of flat connections. We then explain the notion of the enclosed volume of periodic surfaces $f$, and
derive some identities related to the Hodge decomposition of its differential $df$.

\subsection{Periodic surfaces and translational lattices}
\label{subs:cmcr3}
We begin by considering a conformally immersed CMC 1 surface $f\colon\tilde{\Sigma}\to \mathbb{R}^3\cong \mathfrak{su}(2)$ (assuming $H=1$ after scaling) that may have translational periods. From such an immersion,  we derive the 
associated family of flat $\mathrm{SL}(2,\C)$-connections on the trivial bundle 
$\C^2 \to \Sigma$. This
integrable systems approach to CMC surfaces is classical (cf.\ \cite{B}, \cite{Hi}). Since many of our later 
computations rely on the normalizations, we spell out the precise formula for the family and provide its derivation for completeness.
\medskip

The Euclidean metric on $\mathfrak{su}(2)$ is given by 
$-\tfrac{1}{2}\tr,$ i.e.,  
$$<\xi,\nu>:=-\tfrac{1}{2}\tr(\xi\nu) \quad \text{ for } \quad  \xi,\nu\in\mathfrak{su}(2).$$
Let $f\colon\Sigma\to\mathfrak{su}(2)$ be a conformal immersion with mean curvature $H=1$. Write the differential as an 
$\mathfrak{su}(2)$-valued one-form,
and define $\Phi\in\Omega^{(1,0)}(\Sigma,\mathfrak{sl}(2,\C))$ and 
$\Phi^*\in\Omega^{(0,1)}(\Sigma,\mathfrak{sl}(2,\C))$ by
\begin{equation}\label{eq:defPhifromf}
\Phi+\Phi^* := \tfrac{i}{2}\, df \ \in \Omega^1(\Sigma,i\,\mathfrak{su}(2)).
\end{equation}
Because $f$ is conformal, the $(1,0)$-part $\Phi$ is nilpotent and nowhere 
vanishing. Moreover, since $i df$ is hermitian symmetric, the $(0,1)$-part is $\Phi^*$, the adjoint of $\Phi$ with respect to the standard 
Hermitian inner product on $\C^2$.

\medskip

\begin{Pro}\label{Prop:family_from_f}
Let $f$ be a conformally immersed CMC surface of mean curvature $H=1$ in
$\mathbb R^3 \cong \mathfrak{su}(2)$ with translational periods, and define
$\Phi,\Phi^*$ by \eqref{eq:defPhifromf}.
Then the connection
\[
\nabla := d - \Phi + \Phi^*
\]
is unitary, and the one-parameter family of connections
\begin{equation}\label{eq:famconn-new}
\nabla^\lambda
:=
d + (\lambda^{-1}-1)\Phi - (\lambda-1)\Phi^*,
\qquad \lambda \in \C^*,
\end{equation}
is flat for all $\lambda \in \C^*$, unitary for $\lambda \in\S^1$, and satisfies
$\nabla^1 = d$.
\end{Pro}

\begin{proof}
Since $f$ has only translational periods, the differential
$df \in \Omega^1(\Sigma,\mathfrak{su}(2))$ is well defined.
Using \eqref{eq:defPhifromf}, the one-form $-\Phi+\Phi^*$ is also
$\mathfrak{su}(2)$-valued, so
\[
\nabla = d - \Phi + \Phi^*
\]
is a unitary connection.
If $\lambda \in\S^1$, the form $(\lambda^{-1}-1)\Phi-(\lambda-1)\Phi^*$ is
skew-Hermitian, and hence $\nabla^\lambda$ is unitary for all $\lambda \in\S^1$.
Moreover, $\nabla^{\lambda=1}=d$ by construction.
Since $df$ is closed, we have
\[
d(\Phi+\Phi^*)=0.
\]

\medskip

For $\xi,\nu \in \mathfrak{su}(2)$, the vector product satisfies
\[
\xi \times \nu = \tfrac12[\xi,\nu].
\]
Writing locally $dz=dx+idy$, $df=f_xdx+f_ydy$, and $\Phi=\varphi dz$,
$\Phi^*=\varphi^* d\bar z$, we obtain
\[
f_x=-2i(\varphi+\varphi^*) \qquad \text{and} \qquad
f_y=2\varphi-2\varphi^*.
\]
Let $N$ denote the unit normal of $f$ and let $dA \in \Omega^2(\Sigma)$ be the
induced area form.
Then
\begin{equation}
\begin{split}
2N\,dA
&=2(f_x \times f_y)\,dx \wedge dy
=[f_x,f_y]\,dx \wedge dy \\
&=4[-i(\varphi+\varphi^*),\varphi-\varphi^*]\,dx \wedge dy
=-4[\Phi \wedge \Phi^*].
\end{split}
\end{equation}
For a conformal immersion into Euclidean three-space, the mean curvature $H$
satisfies
\[
d*df = 2H N\, dA.
\]
Since $H=1$ by assumption, this gives
\[
d(2\Phi-2\Phi^*) = d*df = 2N\,dA = -4[\Phi \wedge \Phi^*].
\]
Together with $d(\Phi+\Phi^*)=0$, we obtain
\[
d^\nabla \Phi = d\Phi + [\Phi^* \wedge \Phi] = 0,
\qquad
d^\nabla \Phi^* = d\Phi^* - [\Phi \wedge \Phi^*] = 0.
\]

\medskip

For $\lambda \in \C^*$, consider $\nabla^\lambda$ as in
\eqref{eq:famconn-new}.
Since $\Sigma$ is complex one-dimensional, we have
$$\Phi \wedge \Phi = 0 = \Phi^* \wedge \Phi^*,$$ and hence the curvature
\[
F^{\nabla^\lambda}
=\lambda^{-1} d^\nabla \Phi + F^\nabla - [\Phi \wedge \Phi^*]
- \lambda d^\nabla \Phi^*
= F^\nabla - [\Phi \wedge \Phi^*]
\]
is independent of $\lambda$.
As $F^\nabla - [\Phi \wedge \Phi^*] = F^d = 0$, the connection
$\nabla^\lambda$ is flat for all $\lambda \in \C^*$.
\end{proof}

\begin{Rem}
Reversing the arguments in the proof of Proposition~\ref{Prop:family_from_f},
one may start with a family of flat connections
$\nabla^\lambda = \nabla + \lambda^{-1}\Phi - \lambda \Phi^*$
as in Proposition~\ref{Prop:family_from_f}.
If $\nabla^\lambda$ is unitary for $\lambda \in\S^1$ and trivial at $\lambda = 1$,
then integrating $-2i(\Phi+\Phi^*)$ yields a conformal immersion of constant mean
curvature $H=1$.
\end{Rem}

\begin{Rem}[Recovering the immersion: the Sym--Bobenko formula]\label{Rem:scaling2}
Let $F^\lambda$ denote a parallel frame for $\nabla^\lambda$ with
$F^\lambda(p_0)=\mathrm{Id}$.
Differentiating with respect to $\lambda=e^{i\tau}$ at $\tau = 0$ and using \eqref{eq:famconn-new}, we obtain
\[
0
=
\frac{\partial}{\partial\tau}\Big|_{\tau=0}
\bigl(\nabla^{e^{i\tau}} F^{e^{i\tau}}\bigr)
=
d\dot F - i\Phi - i\Phi^*.
\]
Thus
\[
\dot F = i \partial_\lambda |_{\lambda = 1}F^\lambda = -\tfrac 12 f
\]
(up to translation), recovering the classical Sym--Bobenko formula
\cite[Theorem~1.2]{Bo}.
In particular, the CMC immersion $f$ is encoded directly in the
$\lambda$-derivative of the parallel frame at $\lambda=1$.
\end{Rem}
\begin{Rem}[Area formula]\label{rem:area-formula}
The area of $f$ on the fundamental domain $\Sigma$ can be expressed directly
in terms of the Higgs field $\Phi$:
\[
\mathcal A(f)=2i\int_\Sigma \mathrm{tr}(\Phi\wedge\Phi^*).
\]
This follows immediately from the identity
$df=-2i(\Phi+\Phi^*)$ and from the fact that the induced conformal metric is
\[
-\tfrac12\,\mathrm{tr}\bigl((2i(\Phi+\Phi^*))^2\bigr)
=
2\,\mathrm{tr}(\Phi\Phi^*+\Phi^*\Phi).
\]
\end{Rem}

\subsection{Definition and well-posedness of enclosed volume}
Let $\wt\Sigma \to \Sigma$ be the universal covering of the Riemann surface $\Sigma$.
Let $f \colon \wt\Sigma \to \mathfrak{su}(2)$ be a smooth map with translational periods,
and assume that these periods are contained in a lattice $\Gamma < \mathfrak{su}(2)$ of rank~$3$.
In particular, the map $f \colon \Sigma \to \mathfrak{su}(2)/\Gamma$ and the $1$-form
$df \in \Omega^1(\Sigma,\mathfrak{su}(2))$ are well defined.
By abuse of notation, we do not distinguish between the induced periodic map on $\Sigma$
and its lift to $\wt\Sigma$.

In the following, we always assume that $\Sigma$ is connected and that there
exists a compact oriented $3$-manifold $B$ with $\partial B=\Sigma$ together with
a smooth extension $\hat f \colon B \to \mathfrak{su}(2)/\Gamma$ of $f$.
This condition is equivalent to the vanishing of the induced homology class
$[f] \in H_2(\mathfrak{su}(2)/\Gamma,\Z)$; see, for example,
\cite[Appendix]{LM}.
Note that this assumption does not hold for arbitrary maps, e.g., for a planar
$2$-torus embedded in a $3$-torus.
For CMC surfaces with $H \neq 0$, however, this condition holds automatically;
see also \cite[Lemma~2.3]{RiRo} for the embedded case.
We include a short proof for the general situation.

\begin{Lem}
Let $\Sigma$ be a compact surface and let $\Gamma \subset \mathfrak{su}(2)$ be a
lattice.
Let $f \colon \Sigma \to \mathfrak{su}(2)/\Gamma$ be a closed CMC surface with
$H \neq 0$.
Then $[f]=0 \in H_2(\mathfrak{su}(2)/\Gamma,\Z)$, i.e., there exists a 3-manifold $B,$ a smooth map $\hat f \colon B \to \mathfrak{su}(2)/\Gamma,$ and  $f=\hat f|_{\partial B}$
with
$\partial B=\Sigma$.
If $f$ is embedded, then it bounds a compact $3$-manifold.
\end{Lem}

\begin{proof}
Assume that $[f]\neq 0$.
Then there exists a class
$\omega \in H^2_{\mathrm{dR}}(\mathfrak{su}(2)/\Gamma,\R)$ such that
$f^*\omega \neq 0 \in H^2(\Sigma,\R)$.
Without loss of generality, we may take $\omega = dx_1 \wedge dx_2$, where
$(x_1,x_2,x_3)$ are the Euclidean coordinates on $\mathfrak{su}(2)$.
On the other hand, the CMC condition implies
$d*df = H\, df \wedge df$.
Hence
\[
f^*\omega = df_1 \wedge df_2 = \tfrac{1}{H} d*df_3,
\]
which is exact, a contradiction.
Therefore $[f]=0 \in H_2(\mathfrak{su}(2)/\Gamma,\Z)$, and $f$ extends to a smooth
map $\hat f$ on a $3$-manifold $B$ with $\partial B=\Sigma$.
\end{proof}

We denote the induced volume form on $\mathfrak{su}(2)\cong\R^3$ by $\det.$
We now show that the enclosed volume of a periodic surface is well-defined.

\begin{Lem}\label{lem:evT3}
Let $\Gamma$ be a rank 3 lattice, $\Sigma$ be a compact surface, and $f\colon \Sigma\to\mathfrak{su}(2)/\Gamma$ be smooth. Let $B$ be a 3-manifold with boundary $\Sigma,$ and 
$\hat f\colon B\to \mathfrak{su}(2)/\Gamma$ be an extension of $f$. Then,
the enclosed volume 
\[\mathcal V(f):=\int_B\hat f^*\det \in\; \R\mod \vol(\mathfrak{su}(2)/\Gamma)\]
is well-defined, i.e., independent of the choice of $B$ and $\hat f.$

If the translational periods are contained in a  lattice $\underline\Gamma\subset\Gamma$ of rank 2, and
$\hat f\colon B\to\mathfrak{su}(2)/\underline \Gamma$ extends $f$, then 
\[\mathcal V(f):=\int_B\hat f^*\det\in\R\]
is independent of the choice of $B$ and $\hat f$  with periods contained  entirely in $\underline\Gamma$.
\end{Lem}
\begin{proof}
Let $\tilde B$ be another oriented $3$-manifold with
$\partial \tilde B=(\Sigma,-o)$ (where $-o$ denotes the opposite orientation),
and let $\tilde f\colon \tilde B\to \mathfrak{su}(2)/\Gamma$ be another extension.
Consider the closed oriented $3$-manifold
$M:=B\cup_\Sigma\tilde B$
and the map
$F:=\hat f\cup \tilde f\colon M\to\mathfrak{su}(2)/\Gamma$.
Then
\[
\int_M F^*\det
=
\deg(F)\,\vol(\mathfrak{su}(2)/\Gamma),
\]
which proves the first statement.

For the second part, the difference between the two values of
$\int_B \hat f^*\det$ obtained from different choices of extensions
can be written as
\[
\int_M dh\wedge\eta,
\]
where $M$ is as above, $h\colon M\to\R$ is the (oriented) height function with
respect to a direction transverse to the plane spanned by $\underline\Gamma$, and
$\eta\in\Omega^2(M)$ is closed.
Hence, the statement follows from Stokes' theorem.
\end{proof}
\begin{Rem}
What we call the enclosed volume is also referred to as the algebraic volume in the surface theory 
literature.
\end{Rem}

\begin{Con*}
In the following, we always split
$df\in\Omega^1(\Sigma,\mathfrak{su}(2))$ into its harmonic part and its exact part
\begin{equation}\label{eq:hodge-splitting}df=\psi+d\xi,\end{equation}
where  $\xi\colon \Sigma\to \mathfrak{su}(2)$ is globally well-defined and $\psi$ is harmonic.
\end{Con*}

In particular, the periods of $df$ equal the periods of $\psi.$ Furthermore, if 
$df$ extends to $B$ as a closed 1-form $\omega\in\Omega^1(B,\mathfrak{su}(2))$, then the harmonic part
$\psi$ and the exact part $d\xi$ extend as closed, respectively exact, 1-forms $\hat\psi$ and $d\hat\xi$ to $B$ as well. In fact, there is a smooth map $\hat\xi\colon B\to\mathfrak{su}(2)$ which coincides with $\xi$ on $\Sigma=\partial B,$ and $\hat\psi=\omega-d\hat\xi$ is the closed 1-form extension of the harmonic 1-form 
$\psi.$

\begin{Lem}\label{lem:vol}
Let $f\colon\tilde\Sigma\to\mathfrak{su}(2)$ be a surface with translational periods contained in a lattice $\Gamma$ of rank 3, and  let
$d f=\psi+d\xi\in\Omega^1(\Sigma,\mathfrak{su}(2))$ be the Hodge decomposition of its differential.

Let $B$ be a compact 3-manifold with $\partial B=\Sigma,$ and
$\hat\psi\in\Omega^1(B,\mathfrak{su}(2))$ be an extension of $\psi$ as a closed 1-form with periods in $\Gamma$.
Then,  the enclosed volume of $f\colon\Sigma\to\mathfrak{su}(2)/\Gamma$ is
\begin{equation}\label{formula:vol}
\mathcal V(f)=-\tfrac{1}{12}\int_B\tr(\hat\psi\wedge\hat\psi\wedge \hat\psi)-\tfrac{1}{12}\int_\Sigma\tr(\xi d\xi\wedge d\xi)-\tfrac{1}{4}\int_\Sigma\tr(\xi\psi\wedge \psi)-\tfrac{1}{4}\int_\Sigma\tr(\psi \wedge d \xi \xi),\end{equation} 
and takes values in $\R\mod \vol(\mathfrak{su}(2)/\Gamma)$. \end{Lem}


\begin{proof}
Consider the map $\hat f\colon B\to\mathfrak{su}(2)/\Gamma$.
A direct computation gives that 
\[-\tr(d\hat f\wedge d\hat f\wedge d\hat f)
=12\hat f^*\vol.\]
On the other hand, decomposing $d\hat f=\hat\psi+d\hat\xi$,
where $\hat\xi\colon B\to\mathfrak{su}(2)$ is an arbitrary extension of $\xi$ to $B$,
we obtain
\begin{equation}
\begin{split}
d\hat f\wedge d\hat f\wedge d\hat f
=&\;\hat\psi\wedge\hat\psi\wedge \hat\psi
+\hat\psi\wedge d\hat \xi\wedge \hat\psi
+d\hat\xi\wedge \hat\psi\wedge \hat\psi
+d\hat\xi\wedge d\hat\xi\wedge \hat\psi\\
&+\hat\psi\wedge\hat\psi\wedge d\hat \xi
+\hat\psi\wedge d\hat \xi\wedge d\hat \xi
+d\hat\xi\wedge \hat\psi\wedge d\hat \xi
+d\hat\xi\wedge d\hat\xi\wedge d\hat \xi.
\end{split}
\end{equation}
Using the cyclicity of the trace, this yields
\begin{equation}
\begin{split}
-12\int_B\hat f^*\vol
=&\int_B\tr(\hat\psi\wedge\hat\psi\wedge \hat\psi)
+\int_\Sigma\tr(\xi\, d\xi\wedge d\xi)
+3\int_\Sigma\tr(\xi\psi\wedge \psi)
+3\int_\Sigma\tr(\psi \wedge d \xi\, \xi)
\end{split}
\end{equation}as claimed.
\end{proof}

Let $f$ be a CMC surface in $\mathfrak{su}(2)$ with translational periods contained in $\Gamma$.
Let
\[
\rho \colon H_1(\Sigma,\mathbb{Z}) \to \Gamma
\]
denote the period map of $df$, equivalently of its harmonic part $\psi$. In analogy with the case of CMC surfaces in $\mathbb{S}^3$ (see, for example,
\cite{Hi2,HHTS3}), we define
\[
\mathcal V_\Sigma(\rho)
:=
-\tfrac{1}{12}\int_B \tr(\hat\psi \wedge \hat\psi \wedge \hat\psi)
\]
to be the Wess--Zumino--Witten term of $\rho$ with respect to the Riemann surface
$\Sigma$.
As for the enclosed volume $\mathcal V(f)$, the quantity $\mathcal V_\Sigma(\rho)$
is well defined only up to an integer multiple of
$\vol(\mathfrak{su}(2)/\Gamma)$; however, their difference
$\mathcal V(f)-\mathcal V_\Sigma(\rho)$ is a well-defined real number.

As before, let $\partial B = \Sigma$ and assume that $df$ is the restriction of a
closed $1$-form $\omega \in \Omega^1(B,\mathfrak{su}(2))$.
Suppose further that the image of the period map $\rho$ is contained in a lattice
$\underline{\Gamma}$ of rank~$2$.
Then the component of $\psi$ orthogonal to $\underline{\Gamma}$ vanishes. Without loss of generality, we choose an extension
$\hat{\psi} \in \Omega^1(B,\mathbb{R}\,\underline{\Gamma})
\subset \Omega^1(B,\mathfrak{su}(2))$
taking values in the two-dimensional subspace spanned by
$\underline{\Gamma}$.
We will always make this choice whenever the period lattice has rank at most~$2$.

\begin{Cor}\label{Cor:planarvol}
With the same notation as above, assume that the lattice
$\underline{\Gamma} \subset \Gamma$ spanned by the periods of $\tilde f$ has
rank at most~$2$.
Then the enclosed volume of
$f \colon \Sigma \to \mathfrak{su}(2)/\underline{\Gamma}$
is well defined in $\mathbb{R}$ and satisfies
\[
\mathcal V(f)
=
-\tfrac{1}{12}\int_\Sigma \tr(\xi\, d\xi \wedge d\xi)
-\tfrac{1}{4}\int_\Sigma \tr(\xi\, \psi \wedge \psi)
-\tfrac{1}{4}\int_\Sigma \tr(\psi \wedge d\xi\, \xi)
\in \mathbb{R}.
\]
\end{Cor}

\begin{proof}
Since the extension
$\hat\psi \in \Omega^1(B,\mathfrak{su}(2))$
takes values in a two-dimensional linear subspace, we have
\[
\tr(\hat\psi \wedge \hat\psi \wedge \hat\psi) = 0.
\]
The result follows from Lemma~\ref{lem:vol}.
\end{proof}

\section{Gauge theoretic derivation of the Minkowski formula}\label{sec:Kenclosed}

For closed CMC surfaces in $\mathbb{R}^3$, the classical Minkowski identity
\[
\mathcal A(f)=3H\,\mathcal V(f)
\]
relates area and enclosed volume in a simple and rigid way.
For periodic CMC surfaces, the enclosed volume requires additional global input.

The gauge-theoretic description of CMC surfaces (as summarized in
Section~\ref{subs:cmcr3}) provides a natural framework for formulating such a
relation. Any CMC immersion gives rise to a family of flat connections
\[
\lambda \longmapsto \nabla^\lambda,
\]
unitary on $\S^1$ and trivial at the Sym point $\lambda=1$.
Many geometric quantities of the immersion -- such as the area and the Hopf
differential  -- admit gauge-invariant expressions in terms of this family.
The enclosed volume, however, is governed by higher-order information of the
associated family and cannot be recovered from first-order data alone.

From the viewpoint of the moduli space of flat connections, the path
$\tau \mapsto \nabla^{e^{i\tau}}$
traces a curve emanating from the trivial connection. When the immersion is compact, this curve is constant at first order, which is precisely what leads to the classical Minkowski identity. For periodic surfaces, first-order triviality fails, and the second-order behaviour of the curve encodes additional global information. This second-order information is captured by a natural scalar quantity:
$$K \;=\; \frac{i}{2\pi}\int_\Sigma \mathrm{tr}\big(D' \wedge D''\big),$$
where $D'$ and $D''$ are respectively the first and second derivatives (in a harmonic gauge) of the associated family at $\lambda=1.$ The integrand is reminiscent of the curvature term appearing in the
Chern--Simons description of the moduli space of flat connections; indeed,
$K$ can be interpreted as the infinitesimal holonomy of the Chern--Simons
connection along the loop traced out by $\nabla^\lambda$.
Crucially, $K$ is gauge invariant and depends only on the gauge class of the
family $\nabla^\lambda$, and not on the choice of immersion (in the presence of
non-trivial isospectral deformations).

The key observation is that the classical Minkowski formula admits a natural
extension once this curvature term is taken into account.
The area, the enclosed volume, and a Wess--Zumino--Witten--type term associated with
the harmonic part of $df$, together with the curvature $K$, satisfy a linear
relation which reduces to the classical Minkowski identity in the closed case.
This provides a fully gauge-theoretic description of the enclosed
volume for CMC surfaces with translational periods.

\subsection{Associated families and harmonic gauges}\label{sec:curvature_prelim}

Let $\mathrm{G}\in\{\mathrm{SU}(2),\mathrm{SL}(2,\C)\},$ and $\mathfrak{g}$ be its Lie algebra.
We consider smooth families of flat $\mathrm{G}$-connections 
\[\tau\longmapsto\nabla^\tau\]
depending on some real parameter $\tau\in(-\epsilon,\epsilon)$ with $\nabla^{\tau=0}=d.$

\begin{Lem}\label{lem:exharmonicgauge}
There exists a smooth family of gauge transformations
$\tau \mapsto g(\tau) \in \Gamma(\Sigma,\mathrm{G})$
such that
\[
\left.\frac{d}{d\tau}\right|_{\tau=0} (\nabla^\tau . g(\tau))
\in \Omega^1(\Sigma,\mathfrak{g})
\]
is a harmonic $\mathfrak{g}$-valued $1$-form.
We call $g$ a \emph{harmonic gauge}.
\end{Lem}

\begin{proof}
Expand $\nabla^\tau = d + \tau \phi + O(\tau^2)$.
Since $\nabla^\tau$ is flat for all $\tau$, the $1$-form $\phi$ is closed, i.e.,
$d\phi = 0$.
Decompose
\[
\phi = \psi + d\xi
\]
into its harmonic part $\psi$ and its exact part $d\xi$, for some well-defined
map $\xi \colon \Sigma \to \mathfrak{g}$.
Define
\[
g(\tau) := \exp(-\tau \xi).
\]
Then
\[
\nabla^\tau . g(\tau)
= d + \tau \phi - \tau d\xi + O(\tau^2)
= d + \tau \psi + O(\tau^2),
\]
which has a harmonic derivative at $\tau = 0$.
\end{proof}

\subsection{Definition and invariance of the curvature $K$}\label{sec:curvature}
The following observation is crucial for computing the enclosed volume of
periodic CMC surfaces in $\mathbb{R}^3$.
\begin{The}\label{thm:defK}
Consider a family of flat connections $\tau \mapsto \nabla^\tau$, and let
$\tau \mapsto g(\tau)$ be a harmonic gauge as in
Lemma~\ref{lem:exharmonicgauge}.
Set
\[
D(\tau) := \nabla^\tau . g(\tau), \qquad
D' := \left.\frac{d}{d\tau}\right|_{\tau=0} D(\tau), \qquad
D'' := \left.\frac{d^2}{d\tau^2}\right|_{\tau=0} D(\tau).
\]
Then
\[
K := \tfrac{i}{2\pi} \int_\Sigma \tr(D' \wedge D'')
\]
is independent of the choice of harmonic gauge $g$.
We call $K$ the curvature of the family $\nabla^\tau$ at $\tau = 0$.

If $\tau \mapsto \tilde{\nabla}^\tau$ is another family of flat connections with
$\tilde{\nabla}^{\tau=0}=d$ that is gauge equivalent to $\nabla^\tau$, i.e.,
\[
\tilde{\nabla}^\tau = \nabla^\tau \cdot h(\tau)
\]
for some smooth family of gauge transformations
$\tau \mapsto h(\tau)$, then the corresponding curvature invariant
$\tilde K$ coincides with $K$.
\end{The}
\begin{proof}
Assume that $\nabla^\tau = d + \tau \psi + \tau^2 \phi + O(\tau^3)$ has harmonic
derivative $\psi$ at $\tau = 0$, and let $\tau \mapsto g(\tau)$ be a gauge such
that $\nabla^\tau . g(\tau)$ still has harmonic derivative.
By the standard properties of the trace (in particular, invariance under
conjugation), we may assume without loss of generality that $g(0)=\Id$. Consider the expansion
\[
g(\tau) = \Id + \xi_1 \tau + \xi_2 \tau^2 + \dots .
\]
Expanding $\nabla^\tau . g(\tau)$ in $\tau$ gives
\[
\nabla^\tau . g(\tau)
=
d + \tau(\psi + d\xi_1)
+ \tau^2(\phi + d\xi_2 + [\psi,\xi_1] - \xi_1 d\xi_1)
+ O(\tau^3).
\]
Since $\psi$ is harmonic and $d\xi_1$ is orthogonal to the space of harmonic
$1$-forms, we have $d\xi_1=0$, and hence $\xi_1 = C$ is constant.
Moreover, expanding the flatness condition of $\nabla^\tau$ in $\tau$ yields
\[
d\psi = 0
\qquad \text{and} \qquad
d\phi + \psi \wedge \psi = 0.
\]
Therefore,
\[
\tr(\psi \wedge [\psi,\xi_1])
= 2\,\tr(\psi \wedge \psi\, C)
= -2\, d\,\tr(\phi\, C),
\]
and the first claim follows from
\[
\int_\Sigma \tr\bigl(\psi \wedge (d\xi_2 + [\psi,\xi_1])\bigr)
=
-\int_\Sigma d\,\tr(\psi\, \xi_2 + 2\phi\, C)
= 0.
\]
The second statement follows from the first.
\end{proof}

The motivation for the definition of $K$ stems from the following observations.
The moduli space $\mathcal M^{\mathrm{irr}}_{SU(2)}(\Sigma)$ of irreducible flat
$\mathrm{SU}(2)$-connections on a compact Riemann surface $\Sigma$ carries a
natural Riemannian metric defined as follows.
The tangent space to the space of flat unitary connections at a connection
$\nabla$ is given by the (infinite-dimensional) space of
$d^\nabla$-closed $\mathfrak{su}(2)$-valued $1$-forms $\xi$.
The non-degenerate inner product
\[
\langle \xi,\mu \rangle
:=
\tfrac{1}{2\pi}\int_\Sigma \tr(\xi \wedge *\mu)
\]
is invariant under the action of the gauge group, which acts by conjugation on
tangent vectors.
The vertical subspace is given by the image of
\[
d^\nabla \colon \Omega^0(\Sigma,\mathfrak{su}(2))
\to \Omega^1(\Sigma,\mathfrak{su}(2)),
\]
while its finite-dimensional orthogonal complement is the horizontal space
\[
\mathcal H_\nabla
=
\bigl\{
\xi \in \Omega^1(\Sigma,\mathfrak{su}(2))
\mid d^\nabla \xi = 0 = d^\nabla *\xi
\bigr\}
\]
of harmonic $1$-forms.
The unitary gauge group maps horizontal spaces isometrically to one another.
Consequently, the horizontal space identifies with the tangent space
$T_{[\nabla]}\mathcal M^{\mathrm{irr}}_{SU(2)}(\Sigma)$, which is thereby equipped
with a scalar product.
This endows $\mathcal M^{\mathrm{irr}}_{SU(2)}(\Sigma)$ with the structure of a
Riemannian manifold; in fact, it is K\"ahler, with integrable complex structure
$J=-*$ on $\mathcal H_\nabla$.
The corresponding (imaginary-valued) K\"ahler form is the Atiyah--Bott symplectic
form $\Omega$, which is the curvature of the natural connection $\mathcal D$ on
the Chern--Simons line bundle
$\mathcal L \to \mathcal M^{\mathrm{irr}}_{SU(2)}(\Sigma)$. Note that the natural projection from the space of flat irreducible connections
to the moduli space
$\mathcal M^{\mathrm{irr}}_{SU(2)}(\Sigma)$
is a Riemannian submersion.
Let $\gamma$ be a curve in the moduli space, and let $\hat\gamma$ be a lift of
$\gamma$ to the space of flat unitary connections that is horizontal at
$\tau_0$.
Denote by $D$ the Levi--Civita connection on
$\mathcal M^{\mathrm{irr}}_{SU(2)}(\Sigma)$.

\begin{Pro}\label{pro:KcurvatureLevi}
With the above notation, the following equality holds at $\tau_0$:
\[
\Omega\bigl(\gamma' \wedge \nabla_{\gamma'}\gamma'\bigr)
=
\tfrac{i}{2\pi}\int_\Sigma \tr(\hat\gamma' \wedge \hat\gamma'').
\]
\end{Pro}
\begin{proof}
The curve $\hat\gamma$ maps into the submanifold of flat irreducible unitary connections inside the affine space of unitary connections. We decompose \[\hat\gamma''=h\oplus x\oplus  n,\]
where $n$ is in the normal bundle (and $d^\nabla$-coexact), $x$ is vertical (and $d^\nabla$-exact) and $h$ is horizontal (i.e. $d^\nabla$-harmonic). As in finite dimensions, see for example \cite[XIV, \textsection 1 and \textsection 3]{Lang}, 
we obtain $h=D_{\gamma'}\gamma'$ after identification of the horizontal space with the tangent space of the moduli space. Because $\hat\gamma'$ is horizontal (at $\tau_0$) we obtain by Stokes' Theorem
$\int\tr(\hat\gamma'\wedge x)=0=\int\tr(\hat\gamma'\wedge n)$ (at $\tau_0$). The result follows.
\end{proof}

Assume that for all sufficiently small $\tau$ there exists a unique sufficiently short
geodesic from the gauge class of $\nabla^\tau$ to the trivial gauge class in the
moduli space of flat $\mathrm{SU}(2)$-connections.
Following the family $\tau \mapsto \nabla^\tau$ and then returning along these
geodesics to the trivial gauge class yields, by abuse of notation, a family of
closed curves in the moduli space of flat unitary connections, parametrized by
$\tau$.
Expanding the holonomy of the natural connection $\mathcal D$ on the
Chern--Simons line bundle
$\mathcal L \to \mathcal M$
in $\tau$ then produces, at cubic order, a term proportional to $K$.

\begin{Rem}
The appearance of the curvature term $K$ is closely analogous to the situation
for minimal and CMC surfaces in $S^{3}$.
In that setting, the moduli space of flat $\mathrm{SU}(2)$-connections carries
the Chern--Simons line bundle, whose natural connection $\mathcal D$ has curvature
equal to the Atiyah--Bott symplectic form.
The holonomy of this connection along the loop defined by the associated family
then encodes the enclosed volume of the surface; see \cite{Hi2} for the minimal
surface case and \cite{HHTS3} for CMC surfaces.
\end{Rem}

\subsection{The proof of the generalized Minkowski formula}\label{ssec:minkowski}

Before deriving the main enclosed volume formula, we establish the following
integral identity.

\begin{Lem}\label{lem:cmcm1id}
Let $\Gamma < \mathfrak{su}(2)$ be a lattice, and let
$f$ be a compact CMC surface in $\mathfrak{su}(2)/\Gamma$
with mean curvature $H=1$.
Let
\[
df = \psi + d\xi
\]
be the Hodge decomposition of its differential.
Then
\[
\int_\Sigma \tr(\psi \wedge *\psi)
=
4\,\mathcal A(f)
+
\int_\Sigma \tr(\xi\, df \wedge df),
\]
where $\mathcal A(f)$ denotes the area of $f$.
\end{Lem}
\begin{proof}
Since $f$ has mean curvature $H=1$, we have
\[
d*df = 2N\, dA = df \wedge df,
\]
where on the right-hand side the matrix product is used.
Let
\[
*df = *\psi + *d\xi
\]
be the decomposition into harmonic and co-exact parts, and set
\[
*df^{T} := *\psi
\qquad \text{and} \qquad
*df^{N} := *d\xi.
\]

Thus,
\[
\int_\Sigma \tr(df \wedge *df^{T})
=
\int_\Sigma \tr\bigl((\psi + d\xi) \wedge *\psi\bigr)
=
\int_\Sigma \tr(\psi \wedge *\psi).
\]
On the other hand, since
\[
d*d\xi = d*(\psi + d\xi) = d*df = df \wedge df,
\]
we also obtain
\begin{equation*}
\begin{split}
\int_\Sigma \tr(df \wedge *df^{T})
&=
\int_\Sigma \tr(df \wedge *df)
-
\int_\Sigma \tr(df \wedge *df^{N}) \\
&=
4\,\mathcal A(f)
-
\int_\Sigma \tr\bigl((\psi + d\xi) \wedge *d\xi\bigr) \\
&=
4\,\mathcal A(f)
-
\int_\Sigma \tr(d\xi \wedge *d\xi) \\
&=
4\,\mathcal A(f)
-
\int_\Sigma d\,\tr(\xi \wedge *d\xi)
+
\int_\Sigma \tr(\xi\, d*d\xi) \\
&=
4\,\mathcal A(f)
+
\int_\Sigma \tr(\xi\, df \wedge df),
\end{split}
\end{equation*}
which proves the lemma.
\end{proof}

By the Minkowski formula (see \cite{Min} or \cite{Hs}), the enclosed volume of a
closed CMC surface in $\mathbb{R}^3 = \mathfrak{su}(2)$ can be expressed in terms
of its area.
The following theorem provides a generalization of the Minkowski formula to the
case of periodic CMC surfaces.

\begin{The}\label{Thm:enclosedvolr3}
Let $\Gamma \subset \mathfrak{su}(2)$ be a lattice, and let
$f \colon \Sigma \to \mathfrak{su}(2)/\Gamma$
be a compact CMC surface with mean curvature $H=1$.
Then the area $\mathcal A(f)$ and the enclosed volume $\mathcal V(f)$ of $f$
satisfy
\[
K
=
-\tfrac{i}{2\pi}\mathcal A(f)
+
\tfrac{3i}{2\pi}\mathcal V(f)
-
\tfrac{3i}{2\pi}\mathcal V_\Sigma(\rho),
\]
where $K$ is the curvature of the associated family
$\tau \mapsto \nabla^{e^{i\tau}}$ defined in
Theorem~\ref{thm:defK}, and $\mathcal V_\Sigma(\rho)$ denotes the
Wess--Zumino--Witten term of the period map $\rho$.

If the lattice $\Gamma$ has rank at most~$2$, then
\[
K
=
-\tfrac{i}{2\pi}\mathcal A(f)
+
\tfrac{3i}{2\pi}\mathcal V(f).
\]
\end{The}

\begin{proof}
Consider the family of flat connections
\[
\nabla^\lambda
=
d + (\lambda^{-1}-1)\Phi - (\lambda-1)\Phi^*
\]
associated to $f$, and restrict it to the unit circle
$\lambda = e^{i\tau}$ for $\tau \in \mathbb{R}$.
Expanding in $\tau$ at $\tau = 0$ gives
\begin{equation*}
\begin{split}
\nabla^{e^{i\tau}}
&=
d - i(\Phi+\Phi^*)\,\tau + \tfrac{1}{2}(-\Phi+\Phi^*)\,\tau^2 + \dots \\
&=
d + \tfrac{1}{2}(d\xi+\psi)\,\tau
- \tfrac{1}{4}(*d\xi+*\psi)\,\tau^2 + \dots ,
\end{split}
\end{equation*}
where the closed $1$-form
\[
df = -2i(\Phi+\Phi^*) = d\xi + \psi
\]
is decomposed into its exact part $d\xi$ and its harmonic part $\psi$. Consider the family of gauge transformations
\[
\tau \longmapsto g(\tau) := \exp\!\left(-\tfrac{\tau}{2}\,\xi\right).
\]
This defines a harmonic gauge for the family
$\tau \mapsto \nabla^{e^{i\tau}}$.
A direct computation yields
\begin{equation}\label{eq:Dexp}
\begin{split}
D(\tau)
:=
\nabla^{e^{i\tau}} . g(\tau)
=
d
+ \tfrac{1}{2}\psi\,\tau
-\tfrac{1}{4}\bigl(
*d\xi + *\psi + [\psi,\xi]
+ d\xi\,\xi - \tfrac{1}{2}d(\xi^2)
\bigr)\tau^2
+ \dots .
\end{split}
\end{equation}

Substituting the terms of \eqref{eq:Dexp} into the definition of $K$ and using
Lemma~\ref{lem:cmcm1id}, we obtain
\begin{equation*}
\begin{split}
8\pi i K
&=
8\pi i\,\Omega(D' \wedge D'') \\
&=
\int_\Sigma
\tr\!\left(
\psi \wedge
\bigl(
*d\xi + *\psi + [\psi,\xi]
+ d\xi\,\xi - \tfrac{1}{2}d(\xi^2)
\bigr)
\right) \\
&=
\int_\Sigma \tr(\psi \wedge *\psi)
+ \int_\Sigma \tr(\psi \wedge [\psi,\xi])
+ \int_\Sigma \tr(\psi \wedge d\xi\,\xi) \\
&=
4\,\mathcal A(f)
+ \int_\Sigma \tr(\xi\, df \wedge df)
+ \int_\Sigma \tr(\psi \wedge \psi\,\xi)
- \int_\Sigma \tr(\psi \wedge \xi\,\psi)
+ \int_\Sigma \tr(\psi \wedge d\xi\,\xi) \\
&=
4\,\mathcal A(f)
- 12\,\mathcal V(f)
+ 12\,\mathcal V_\Sigma(\rho),
\end{split}
\end{equation*}
where the last equality follows from Lemma~\ref{lem:vol} and the definition of the
Wess--Zumino--Witten term $\mathcal V_\Sigma(\rho)$.
\end{proof}

A direct corollary is the classical Minkowski formula \cite{Min} for CMC surfaces
in $\mathbb{R}^3$.

\begin{Cor}\label{cor:encvol-noperiods}
If $f$ is a closed CMC immersion into $\mathbb{R}^3$ with constant mean curvature
$H$, then
\[
\mathcal A(f) = 3H\,\mathcal V(f).
\]
\end{Cor}

\begin{proof}
By scaling, we may normalize to $H=1$.
If $f$ has no translational periods, then $D'=0$ in any harmonic gauge, and hence
$K=0$ by Theorem~\ref{thm:defK}.
On the other hand, $\mathcal V_\Sigma(\rho)$ vanishes automatically when the
translational periods lie in a lattice of rank at most~$2$.
Thus, the Minkowski formula follows as a special case of
Theorem~\ref{Thm:enclosedvolr3}.
\end{proof}

\begin{Exa}
The round sphere of radius $1$ with outward-pointing normal has mean curvature
$H=1$, area $4\pi$, and enclosed volume $\tfrac{4}{3}\pi$.
Theorem~\ref{Thm:enclosedvolr3} therefore implies that $K$ vanishes.
On the other hand, the moduli space of flat $\mathrm{SU}(2)$-connections on the
$2$-sphere consists of a single point, so the Atiyah--Bott symplectic form
$\Omega$ vanishes identically; compare \cite[Example~2]{HHTS3}.
As a consequence, $K=0$ holds automatically.
\end{Exa}

\begin{Exa}\label{ex:flatcylinder}
Consider the round cylinder of radius $\tfrac{1}{2}$ and height $\tfrac{h}{2}>0$
given by
\[
f \colon \Sigma=\mathbb{C}/(h\mathbb{Z}+2\pi i\mathbb{Z})
\to \mathfrak{su}(2)/\Gamma,
\qquad
(x,y)\longmapsto
\begin{psmallmatrix}
 \tfrac{i x}{2} & -\tfrac{1}{2} e^{-i y} \\
 \tfrac{1}{2} e^{i y} & -\tfrac{i x}{2}
\end{psmallmatrix},
\]
where $\Gamma:=\tfrac{i}{2}\diag(1,-1)\,\mathbb{Z}$.
The area is
\[
\mathcal A(f)=\tfrac{h}{2}\pi,
\]
and the enclosed volume is
\[
\mathcal V(f)=\tfrac{h}{8}\pi.
\]
By Theorem~\ref{Thm:enclosedvolr3}, the curvature $K$ satisfies
\begin{equation}\label{exa:hol3}
K
=
-\tfrac{i}{2\pi}\mathcal A(f)
+\tfrac{3i}{2\pi}\mathcal V(f)
=
-\tfrac{i}{4}h
+\tfrac{3i}{16}h
=
-\tfrac{i}{16}h.
\end{equation}

On the other hand, the associated family along the unit circle
$\lambda = e^{i\tau}$ is given by
\begin{equation*}
\begin{split}
\nabla^{e^{i\tau}}
= d
&+
\begin{psmallmatrix}
 \tfrac{1}{8} e^{-i \tau } \left(-1+e^{2 i \tau }\right)
 &
 \tfrac{1}{8} i \left(-1+e^{i \tau }\right)^2 e^{-i (\tau +y)} \\
 \tfrac{1}{8} i \left(-1+e^{i \tau }\right)^2 e^{i (y-\tau )}
 &
 -\tfrac{1}{8} e^{-i \tau } \left(-1+e^{2 i \tau }\right)
\end{psmallmatrix}
dx \\
&+
\begin{psmallmatrix}
 -\tfrac{1}{8} i e^{-i \tau } \left(-1+e^{i \tau }\right)^2
 &
 \tfrac{1}{8} \left(-1+e^{2 i \tau }\right) e^{-i (\tau +y)} \\
 \tfrac{1}{8} \left(-1+e^{2 i \tau }\right) e^{i (y-\tau )}
 &
 \tfrac{1}{8} i e^{-i \tau } \left(-1+e^{i \tau }\right)^2
\end{psmallmatrix}
dy .
\end{split}
\end{equation*}
Gauging by
\[
g =
\begin{psmallmatrix}
 -i \left(-1+e^{\frac{i \tau}{2}}\right) e^{-\frac{i y}{2}}
 &
 -i \left(1+e^{\frac{i \tau}{2}}\right) e^{-\frac{i y}{2}} \\
 \left(1+e^{\frac{i \tau}{2}}\right) e^{\frac{i y}{2}}
 &
 \left(-1+e^{\frac{i \tau}{2}}\right) e^{\frac{i y}{2}}
\end{psmallmatrix},
\]
we obtain
\begin{equation*}
\begin{split}
D^\tau
=
\nabla^{e^{i\tau}} . g
=
d
&+
\begin{psmallmatrix}
 -\tfrac{1}{2} i \sin \!\left(\tfrac{\tau}{2}\right) & 0 \\
 0 & \tfrac{1}{2} i \sin \!\left(\tfrac{\tau}{2}\right)
\end{psmallmatrix}
dx \\
&+
\begin{psmallmatrix}
 \tfrac{1}{2} i \cos \!\left(\tfrac{\tau}{2}\right) & 0 \\
 0 & -\tfrac{1}{2} i \cos \!\left(\tfrac{\tau}{2}\right)
\end{psmallmatrix}
dy .
\end{split}
\end{equation*}
In particular, $g$ is a harmonic gauge for
$\tau \mapsto \nabla^{e^{i\tau}}$, and we compute
\[
K
=
\tfrac{i}{2\pi}\int_\Sigma \tr(D' \wedge D'')
=
-\tfrac{i}{16}h,
\]
in agreement with \eqref{exa:hol3} and
Theorem~\ref{Thm:enclosedvolr3}.
\end{Exa}

\section{Computing the curvature $K$}\label{section:computeK}

\subsection{The CMC torus case}\label{ssec:specdata}
Consider an equivariant CMC $1$ torus with translational periods, and let
$\omega = df \in \Omega^1(T^2,\mathbb{R}^3)$. 
Clearly, the periods of $\omega$ generate a lattice
$\Gamma \le \mathbb{R}^3$ of rank $r \le 2$.
In fact, we have the following result.

\begin{The}\label{pro:1dlat}
The period lattice of an equivariant CMC torus $f$ with $H \neq 0$ and
non-trivial translational periods has rank~$1$.
\end{The}

\begin{proof}
By scaling, we may assume that the mean curvature of the equivariant CMC surface
$f$ is equal to $1$.
The associated family $\nabla^\lambda$ is well defined on the trivial
$\mathbb{C}^2$-bundle over a torus $T^2=\mathbb{C}/\Lambda$ for some lattice
$\Lambda$.
Let $dz = dx + i\,dy$ be a holomorphic $1$-form on $T^2$, and let
\[
df = (\mu\,dx + \nu\,dy) + d\xi
\]
be the Hodge decomposition of
$df \in \Omega^1(T^2,\mathfrak{su}(2))$, where
$\mu,\nu \in \mathfrak{su}(2)$ and
$\xi \colon T^2 \to \mathfrak{su}(2)$.

Expanding at $\tau=0$, we have
\[
\nabla^{e^{i\tau}}
=
d + \tfrac{1}{2} df\,\tau + \tfrac{1}{2}\psi_2\,\tau^2 + \dots .
\]
Expanding the curvature gives
\begin{equation}
\begin{split}
0
=
F^{\nabla^{e^{i\tau}}}
=
\bigl(\tfrac{1}{4}[df \wedge df] + \tfrac{1}{2} d\psi_2 \bigr)\tau^2 + \dots ,
\end{split}
\end{equation}
so that
\[
\tfrac{1}{2}[df \wedge df]
=
[\mu,\nu]\,dx \wedge dy
+
[\mu\,dx \wedge d\xi]
+
[\nu\,dy \wedge d\xi]
+
[d\xi \wedge d\xi]
\]
must be exact.
On the other hand, the exactness of
\[
[\mu\,dx \wedge d\xi]
+
[\nu\,dy \wedge d\xi]
+
[d\xi \wedge d\xi]
\]
follows from the exactness of $d\xi$, and therefore
\[
[\mu,\nu]=0.
\]
This implies that $\mu$ and $\nu$ are linearly dependent, and the claim follows.
\end{proof}

For embedded surfaces, the above theorem follows directly from the results of
\cite{Meeks}.
Conversely, let $\Gamma$ be a lattice of rank~$1$ and let
$f \colon \Sigma \to \mathbb{R}^3/\Gamma$ be a compact, embedded CMC surface.
Then $f$ must be a CMC cylinder by \cite{KKS}.


\subsubsection{Computing $K$ from spectral data}

The associated family of flat connections of a singly periodic CMC cylinder is
well defined on a torus $T^2$ and gives rise to spectral data
$(X,\theta_1,\theta_2,L)$ parametrizing the family
$\lambda \mapsto \nabla^\lambda$; see \cite{Hi,Hi2} for details.
Here, $X$ is a hyperelliptic Riemann surface with real symmetry, and
$\theta_1,\theta_2$ are meromorphic differentials of the second kind, compatible
with the real symmetry and having double poles at the points lying over $0$ and
$\infty$.
If the genus of $X$ is at least $2$, the principal parts of $\theta_1$ and
$\theta_2$ are linearly independent over $\mathbb{R}$.
In this case, they span a cocompact lattice equivalent to the lattice $\Lambda$
of $T^2=\mathbb{C}/\Lambda$.

The hyperelliptic surface $\pi \colon X \to \mathbb{CP}^1$ is not branched over
points on the unit circle.
Hence, the preimage of the unit circle is the disjoint union of two topological
circles.
The differentials $\theta_1$ and $\theta_2$ are the logarithmic differentials of
the eigenvalues $\pm\nu_k$ along the corresponding generators of the lattice
$\Lambda$.
More explicitly, along $\pi^{-1}(\S^1) \subset X$, the unitary gauge class of
$\nabla^\lambda$ is given by
\[
d
+
\diag(\log \nu_1,-\log\nu_1)\,dx
+
\diag(\log \nu_2,-\log\nu_2)\,dy,
\]
where $dx$ and $dy$ are dual to the two generators of $\Lambda$.
At $\lambda=1$, the gauge class is trivial, and we may assume without loss of
generality that for $p \in X$ with $\pi(p)=1$ we have
$\log \nu_1(p)=\log \nu_2(p)=0$.
Note that $dx$ and $dy$ are harmonic.
Since $\pi$ is unbranched over $\S^1$, we may locally use $\lambda = e^{it}$ as a
parameter and expand
\[
\theta_1 = d\log\nu_1 = \bigl(a_1 + a_2(\lambda-1) + \dots\bigr)\,d\lambda,
\qquad
\theta_2 = d\log\nu_2 = \bigl(b_1 + b_2(\lambda-1) + \dots\bigr)\,d\lambda.
\]
Using the definition of $K$ in Theorem~\ref{thm:defK}, we obtain
\begin{equation}\label{exK}
K
=
\tfrac{i}{2\pi}
\int_{\mathbb{C}/\Lambda}
2(a_1 b_2 - a_2 b_1)\,dx \wedge dy
=
\tfrac{i}{\pi}(a_1 b_2 - a_2 b_1)
\int_{\mathbb{C}/\Lambda} dx \wedge dy .
\end{equation}
Together with the energy (area) formula
\cite[Theorem~13.17(ii)]{Hi} and Theorem~\ref{Thm:enclosedvolr3},
\eqref{exK} yields a formula for the enclosed volume $\mathcal V(f)$ in terms of
the spectral data.

 \subsection{DPW potentials}

Since $K$ is gauge invariant (Theorem~\ref{thm:defK}), it can be computed using
different representations of the gauge class of the associated family of flat
connections.
An example of this was given in Section~\ref{ssec:specdata}, where $K$ for an
equivariant CMC torus was computed in terms of the corresponding spectral data.
We now describe a general method for computing $K$ using so-called DPW
potentials.
Throughout this subsection, we restrict to surfaces whose period lattice has
rank~$2$.

We assume that the CMC surface $f$ is described in terms of a
\emph{DPW potential} $d+\eta$.
By definition, a DPW potential is given by a $\lambda$-family (defined on a
punctured disc of radius $r>1$)
\[
\eta=\sum_{k\geq -1}\eta_k\,\lambda^k
\]
of meromorphic $\mathfrak{sl}(2,\mathbb{C})$-valued $1$-forms on $\Sigma$ such that
$\eta_{-1}$ is nilpotent.
The \emph{intrinsic closing conditions} for the DPW potential $\eta$ are the
existence of a holomorphic family of (meromorphic) gauge transformations
$B=b_0+b_1\lambda+\dots$ on $\Sigma$ such that
\[
(d+\eta).B
:=
d + B^{-1}\eta B + B^{-1}dB
=
\nabla + \lambda^{-1}\Phi - \lambda\Phi^*,
\footnote{The differential $d$ acts on the $\Sigma$-variable only.}
\]
for a unitary connection $\nabla$ and a $(1,0)$-form $\Phi$ together with its
adjoint $(0,1)$-form $\Phi^*$.
The poles of $\eta$ are then called \emph{apparent singularities}.

The \emph{extrinsic closing condition} for a translationally periodic CMC surface
is that the monodromy of the unitarized family
$\nabla^\lambda=(d+\eta).B$ is trivial at $\lambda=1$, and that the
$\lambda$-derivative of its monodromy at $\lambda=1$ spans a lattice in
$\mathfrak{su}(2)\cong\mathbb{R}^3$ of rank at most~$3$.
These extrinsic closing conditions imply that $(d+\eta).B$ satisfies the
conditions stated at the beginning of Section~\ref{subs:cmcr3}.
Furthermore, the periods of the CMC surface obtained via
Proposition~\ref{Prop:family_from_f} are contained in a lattice
$\Gamma<\mathfrak{su}(2)$.

\begin{Rem}\label{rem:lgfac}
Assume that $d+\eta$ satisfies the intrinsic and extrinsic closing conditions,
and let $\Phi_{\mathrm{DPW}}$ be a solution of
\begin{equation}\label{eq:dpwsol}
(d+\eta)\,\Phi_{\mathrm{DPW}}=0
\end{equation}
with unitary monodromy at some base point $z_0$.
Then $\Phi_{\mathrm{DPW}}=BF$, where $B$ is the gauge introduced above and
$F=F^\lambda$ is a unitary frame for
$\nabla^\lambda=(d+\eta).B$, as in Remark~\ref{Rem:scaling2}.
On the other hand, the decomposition $\Phi_{\mathrm{DPW}}=BF$ can be obtained
via the loop group Iwasawa decomposition, where $B$ extends holomorphically to
$\lambda=0$ and $F$ is unitary for $\lambda\in\S^1$.
See \cite{DPW} or \cite[Section~3]{Fuchsian} for further details.
Thus, the surface can be reconstructed directly from the DPW potential via
Iwasawa factorization.
\end{Rem}
  
\subsection{Residue formulas for $K$}

We consider a (meromorphic) DPW potential $\eta$ on a compact Riemann surface
$\Sigma$ with apparent singularities at $p_1,\dots,p_n$.
We assume that the intrinsic and extrinsic monodromy problems are solved, i.e.,
that the resulting immersion has translational periods contained in a
three-dimensional lattice.
In particular, the monodromy of $d+\eta$ is trivial at $\lambda=1$.
Gauging by the inverse of $\Phi_{\mathrm{DPW}}^{\lambda=1}$, where
$\Phi_{\mathrm{DPW}}$ is as in \eqref{eq:dpwsol}, we may assume without loss of
generality that $\eta^{\lambda=1}=0$.
Writing $\lambda=e^{i\tau}$, we expand the potential in a neighborhood of
$\tau=0$ as
\begin{equation}\label{eq:eta-exp}
\eta
=
\tau\,\newPhi_1
+
\tfrac{\tau^2}{2}\,\newPhi_2
+
O(\tau^3),
\end{equation}
where $\newPhi_1,\newPhi_2$ are meromorphic
$\mathfrak{sl}(2,\mathbb{C})$-valued $1$-forms. By Lemma~\ref{lem:exharmonicgauge}, there exists a harmonic gauge $g$ such that
\[
(d+\eta).g
=
d
+
\tau\,\psi_1
+
\tfrac{\tau^2}{2}\,\psi_2
+
O(\tau^3),
\]
with $\psi_1$ a smooth, harmonic
$\mathfrak{sl}(2,\mathbb{C})$-valued $1$-form on $\Sigma$.
Since $\eta|_{\lambda=1}=0$, we normalize the harmonic gauge by requiring
$g|_{\lambda=1}=\Id$.

Expand the gauge $g$ at $\lambda=1$, equivalently around $\tau=0$, as
\[
g
=
\Id
+
\tau\,\xi_1
+
\tfrac{\tau^2}{2}\,\xi_2
+
O(\tau^3),
\]
where $\xi_1,\xi_2 \colon \Sigma \to \mathfrak{sl}(2,\mathbb{C})$ are smooth.
Then
\begin{align*}
D^\tau
:=
(d+\eta).g
&=
d
+
(1-\tau\xi_1)
\left(
\tau\,\newPhi_1+\tfrac{\tau^2}{2}\,\newPhi_2
\right)
(1+\tau\xi_1) \\
&\quad
+
(1-\tau\xi_1)
\left(
\tau\,d\xi_1+\tfrac{\tau^2}{2}\,d\xi_2
\right)
+
O(\tau^3) \\
&=
d
+
\tau(\newPhi_1+d\xi_1)
+
\tfrac{\tau^2}{2}
\left(
\newPhi_2+d\xi_2
-2\xi_1 d\xi_1
+2[\newPhi_1,\xi_1]
\right)
+
O(\tau^3).
\end{align*}
We therefore obtain
\begin{equation}\label{eq:Psi}
\begin{cases}
\psi_1=\newPhi_1+d\xi_1,\\[0.3em]
\psi_2=\newPhi_2+d\xi_2-2\xi_1 d\xi_1+2[\newPhi_1,\xi_1].
\end{cases}
\end{equation}

Since $\psi_1$ is a smooth harmonic $1$-form, we have
\[
\Res_{p_j}\newPhi_1
=
\tfrac{1}{2\pi i}
\int_{C(p_j,\varepsilon)} \psi_1
=
0,
\]
where $C(p_j,\varepsilon)$ denotes a positively oriented circle of radius
$\varepsilon>0$ around $p_j$, contained in a coordinate neighborhood of $p_j$.
Hence, for $1\le j\le n$, the primitive
\begin{equation}\label{eq:primpj}
F_1^j=\int \newPhi_1
\end{equation}
defines a meromorphic map in a neighborhood of $p_j$, unique up to an additive
constant.

\begin{Lem}
Assume that the lattice $\Gamma \le \mathfrak{su}(2)$ generated by the periods of
the CMC surface has rank~$2$.
After applying a gauge transformation if necessary, we may assume that the
meromorphic $\mathfrak{su}(2)$-valued $1$-form $\newPhi_1$ is off-diagonal.
\end{Lem}

\begin{proof}
After conjugating by a constant element $C \in \mathrm{SU}(2)$, the lattice
$\Gamma$ is contained in the off-diagonal subspace of $\mathfrak{su}(2)$.
In particular, the primitive $h$ of the diagonal part of $\newPhi_1$ is a
well-defined meromorphic function.
Gauging by $\exp(-\tau h)$ then yields a new $1$-form $\newPhi_1$ which is
off-diagonal.
\end{proof}

In particular, after conjugating by a constant if necessary, the locally defined
primitives $F_1^j$ may be chosen to be off-diagonal.

\begin{The}\label{theorem:computeK}
Assume that $\newPhi_1$ is off-diagonal.
Then the curvature of the family $D^{\tau}$ is given by
\[
K
=
-\sum_{j=1}^n
\tr\!\left(
\Res_{p_j}\bigl(F_1^j\,\newPhi_2\bigr)
\right),
\]
where each primitive $F_1^j$ is chosen to be off-diagonal.
\end{The}

\begin{proof}
By \eqref{eq:Psi}, the diagonal part $d\xi_1^{\diag}$ is harmonic.
Hence $d\xi_1^{\diag}=\psi_1^{\diag}=0$, and $\xi_1^{\diag}$ is constant.
Let
\[
h=\exp(-\tau\xi_1^{\diag}).
\]
Then
\[
gh=\Id+\tau(\xi_1-\xi_1^{\diag})+O(\tau^2).
\]
Since $h$ is constant, $h^{-1}\psi_1 h$ is harmonic and $gh$ is again a harmonic
gauge.
Replacing $g$ by $gh$, we may therefore assume without loss of generality that
$\xi_1$ is off-diagonal.

Then $\xi_1 d\xi_1$ and $[\xi_1,\newPhi_1]$ are diagonal, and by
\eqref{eq:Psi} the off-diagonal parts satisfy
\[
\psi_2^{\off}=d\xi_2^{\off}+\newPhi_2^{\off}.
\]
Since $\psi_2$ is smooth (although not harmonic), we obtain
\[
\int_{C(p_j,\varepsilon)}\newPhi_2^{\off}
=
\int_{C(p_j,\varepsilon)}\psi_2^{\off}
=
O(\varepsilon).
\]
Letting $\varepsilon\to 0$ yields
\[
\Res_{p_j}\newPhi_2^{\off}=0,
\]
so $\newPhi_2^{\off}$ admits a primitive
\[
F_2^j=\int \newPhi_2^{\off},
\]
which defines a meromorphic function in a neighborhood of $p_j$.

Next, we patch the local primitives together using a partition of unity to obtain
globally defined smooth maps $F_1$ and $F_2$.
Let $z=re^{i\theta}$ be a complex coordinate in a neighborhood of $p_j$, and let
$\chi\colon\mathbb{R}\to[0,1]$ be a smooth cutoff function such that $\chi=1$ for
$r\le\varepsilon$ and $\chi=0$ for $r\ge 2\varepsilon$.
Define $\rho_j(z)=\psi(r)$, so that $\rho_j$ is supported in $D(p_j,2\varepsilon)$.
For $k=1,2$, set
\[
F_k=\sum_{j=1}^n \rho_j F_k^j.
\]

Then on $D(p_j,\varepsilon)$ we have
\[
d(\xi_1+F_1)=d\xi_1+\newPhi_1=\psi_1,
\]
which is smooth at $p_j$, so
\[
h_1=\xi_1+F_1
\]
extends smoothly to $\Sigma$.
Similarly, on $D(p_j,\varepsilon)$,
\[
d(\xi_2^{\off}+F_2)=d\xi_2^{\off}+\newPhi_2^{\off}=\psi_2^{\off},
\]
so
\[
h_2=\xi_2^{\off}+F_2
\]
extends smoothly to $\Sigma$.

Since $\psi_1$ is off-diagonal, we obtain on $\Sigma$
\begin{align*}
\tr(\psi_1\wedge\psi_2)
&=\tr\!\left((\newPhi_1+d\xi_1)\wedge(\newPhi_2^{\off}+d\xi_2^{\off})\right)\\
&=\tr\!\left((\newPhi_1-dF_1+dh_1)\wedge(\newPhi_2^{\off}-dF_2+dh_2)\right).
\end{align*}
By Stokes' theorem,
\[
\int_\Sigma\tr\!\left((\newPhi_1-dF_1)\wedge dh_2\right)
=
-\int_\Sigma \tr\,d\!\left((\newPhi_1-dF_1)h_2\right)
=
0,
\]
since $\newPhi_1-dF_1$ is smooth on $\Sigma$.
Similarly,
\[
\int_\Sigma\tr\!\left(dh_1\wedge(\newPhi_2^{\off}-dF_2)\right)
=
\int_\Sigma\tr\,d\!\left(h_1(\newPhi_2^{\off}-dF_2)\right)
=
0,
\]
and
\[
\int_\Sigma\tr(dh_1\wedge dh_2)
=
\int_\Sigma\tr\,d(h_1\,dh_2)
=
0.
\]

Hence, by Stokes' theorem and the residue theorem,
\begin{align*}
K
&=\tfrac{i}{2\pi}\int_\Sigma
\tr\!\left((\newPhi_1-dF_1)\wedge(\newPhi_2^{\off}-dF_2)\right)\\
&=\tfrac{i}{2\pi}\int_\Sigma
\tr\,d\!\left(\newPhi_1 F_2 - F_1\newPhi_2^{\off}-dF_1\,F_2\right)\\
&=-\sum_{j=1}^n
\Res_{p_j}\tr\!\left(F_1^j\newPhi_2^{\off}\right).
\end{align*}
Since $F_1^j$ is off-diagonal, we have
$\tr(F_1^j\newPhi_2^{\diag})=0$, and the result follows.
\end{proof}
\begin{Rem}
The fact that the Wess--Zumino--Witten term vanishes for lattices of rank~$2$ and
that the curvature $K$ can be computed by a residue formula appear to be closely
related.
\end{Rem}
\section{Lawson-type CMC surfaces in $\mathbb{T}^2 \times \mathbb{R}$}
\label{section:lawson}
Let $\Sigma$ be a compact, oriented CMC surface in 
$\mathbb{R}^3/\Gamma,$ where $\Gamma$ is a lattice. Such a surface is called
{\em stable} if it minimizes area up to second order with respect to volume-preserving
variations. Stable periodic CMC surfaces arise naturally as candidates for solutions of the
isoperimetric problem, which seeks area-minimizing surfaces enclosing a prescribed volume;
see \cite{Ros3,HPRR} for background.

It is known that connected, stable, embedded periodic CMC surfaces in $\mathbb{R}^3$ are
topologically rigid: by a result of Ros \cite[Theorem~2]{Ros3}, the genus  of such a surface is bounded by the rank of its period
lattice.
In particular, in the doubly periodic case this bounds the genus by~$2$.
Consequently, Lawson-type surfaces of genus~$2$ represent extremal examples among
embedded doubly periodic CMC surfaces.
We use these surfaces as test cases for the isoperimetric problem in
$\mathbb{T}^2 \times \mathbb{R}$ (see \cite{HPRR} for background on the periodic
isoperimetric problem).
In this section, we describe DPW potentials for Lawson-type CMC surfaces and
compute their enclosed volumes as an application of our generalized Minkowski
formula (Theorem~\ref{Thm:enclosedvolr3}) combined with the curvature formula of
Theorem~\ref{theorem:computeK}.

The Lawson cousins of the Lawson minimal surfaces $\xi_{2,2}$ and $\xi_{3,3}$ in
$\mathbb{S}^3$ give rise to genus 2 CMC surfaces in $\mathbb{T}^2 \times \mathbb{R}$, where
$\mathbb{T}^2$ is respectively a flat equilateral or square torus.
Both belong to one-parameter families of CMC surfaces (see, for example,
\cite{KGB}).
Following Ros~\cite{Ros}, we refer to these as \emph{Lawson-type surfaces}.
The cousins of $\xi_{2,2}$ and $\xi_{3,3}$ are shown in
Figure~\ref{fig:cousin}.

\begin{figure}[b]
\centering
\includegraphics[width=0.35\textwidth]{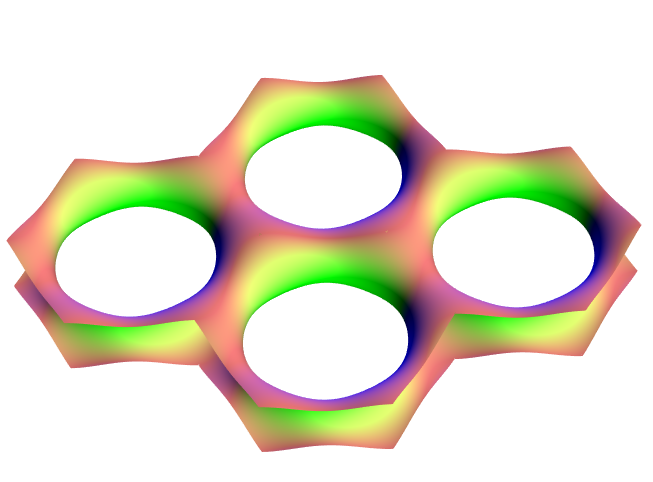}
\hspace{0.75cm}
\includegraphics[width=0.35\textwidth]{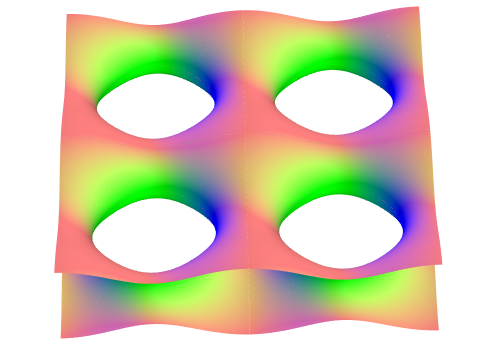}
\hspace{0.25cm}
\caption{%
The conjugate cousin of the Lawson minimal surface $\xi_{2,2}$ in $\mathbb{S}^3$ (left) is a doubly periodic CMC surface in $\mathbb{R}^3$ with hexagonal  period lattice and genus~2 in the quotient.
The conjugate cousin of $\xi_{3,3}$ shown on the right is a doubly periodic CMC surface with square period lattice and genus~2 in the quotient.
The surfaces correspond to the Lawson-type surface \(\Lawson_{3,\tfrac{\pi}{4}}\) and 
\(\Lawson_{4,\tfrac{\pi}{4}}\) in our notation.
Images by the third author using a C++ implementation of the DPW method and the numerical data of Section~\ref{sec:iso}.
}
\label{fig:cousin}
\end{figure}

In this section, we propose DPW potentials for Lawson-type surfaces.
The construction closely follows that of CMC deformations of the Lawson minimal
surfaces $\xi_{1,g}$ in $\mathbb{S}^3$ developed in \cite{HHT2}.
More precisely, we construct a family of CMC surfaces
$\Lawson_{k,\varphi}$ for $k \ge 3$ and $\varphi \in (0,\pi/2)$.
Each surface $\Lawson_{k,\varphi}$ admits a fundamental domain bounded by four
symmetry curves: one lying in a horizontal plane and three lying in vertical
planes meeting at angles $\pi/2$ and $\pi/k$.
When $k \in \{3,4,6\}$, extending the fundamental piece by symmetry yields an
embedded doubly periodic surface;
no further values of $k$ are expected to produce embedded examples,
since only for $k=3,4,6$ do the symmetry angles $\pi/2$ and $\pi/k$ generate a crystallographic
reflection group compatible with a doubly periodic lattice in $\mathbb{R}^3.$
The cases $k=3$ and $k=4$ recover the two Lawson-type families
described above.
For $k = 6$, one obtains an additional family of embedded surfaces of genus 3. As discussed above, these surfaces cannot be stable.
Figure~\ref{fig:cousin5} shows $\Lawson_{6,\tfrac{\pi}{4}}$ which is the conjugate cousin of $\xi_{5,5}$. It was not considered in
\cite{L}.
In general, the parameter $\varphi$ controls the conformal type of the underlying Riemann
surface, as in \cite{HHT2}. Apart from the classical Lawson cousins examples corresponding to $k=3$ and $k=4$
(at $\varphi=\tfrac{\pi}{4}$) and their deformations \cite{KGB}, such families of CMC surfaces in
$\mathbb{T}^2 \times \mathbb{R}$  have  previously been
constructed numerically in \cite{BHS}. 

\begin{figure}[b]
\centering
\includegraphics[width=0.35\textwidth,angle=10]{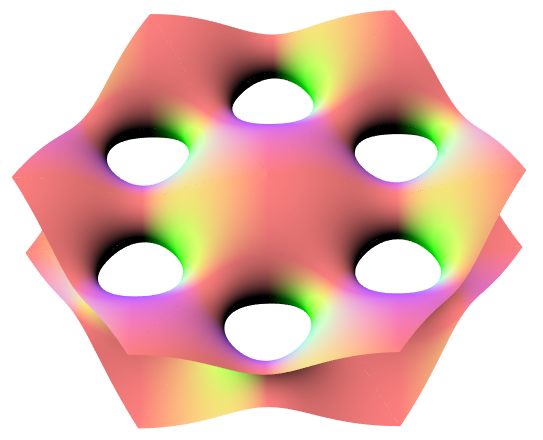}
\caption{%
The conjugate cousin of the Lawson minimal surface $\xi_{5,5}$ in $\mathbb{S}^3$ is a doubly periodic CMC surface with hexagonal period lattice and genus~3 in the quotient.
The surface corresponds to the Lawson-type surface \(\Lawson_{6,\tfrac{\pi}{4}}\)  in our notation.
Image by the third author.
}
\label{fig:cousin5}
\end{figure}

Our approach produces solutions to the monodromy problem rigorously for large values of $k$.
Numerical implementations indicate that the theory continues to behave well down to $k=3$ and $k=4,$ which are the cases of primary interest here, and even to $k=2,$ where one expects to recover the Clifford torus.
In the case $k=2,$ the numerically computed area and enclosed volume agree with the known values for the Clifford torus to within high numerical accuracy.

\subsection{ DPW potentials and closing conditions}\label{sec:dpwlawsonpot}

We fix a parameter $\varphi\in(0,\tfrac{\pi}{2})$ and consider the $4$-punctured
sphere
\[
\Sigma=\mathbb{CP}^1\setminus\{p_1,p_2,p_3,p_4\}
\]
with
\[
p_1=e^{i\varphi},\quad
p_2=-e^{-i\varphi},\quad
p_3=-e^{i\varphi},\quad
p_4=e^{-i\varphi}.
\]
We consider the genus $k-1$ Riemann surface $\Sigma_k$ defined by the algebraic
equation
\begin{equation}\label{eq:algeqyz}
y^k=\frac{(z-p_1)(z-p_2)}{(z-p_3)(z-p_4)}.
\end{equation}
It admits a $k$-fold covering
\[
\pi\colon\Sigma_k\to\mathbb{CP}^1,\qquad (z,y)\mapsto z,
\]
which is totally branched over $p_1,\dots,p_4$.
Note that the Riemann surface underlying the Lawson surface $\xi_{k,k}$ is given
by a $k$-fold covering of $\Sigma_k$; see \cite{HHSch2} for details.

We consider on $\Sigma$ a DPW potential of the form
\[
\eta(z,\lambda)
=
r\,t\sum_{j=1}^4 x_j(\lambda)\,\mathfrak m_j\,\omega_j(z),
\]
where
\[
t=\tfrac{1}{2k},
\]
$r,x_1,x_2,x_3$ are parameters to be determined,
the matrices $\mathfrak m_j$ are
\[
\mathfrak m_1=
\begin{pmatrix} i&0\\0&-i\end{pmatrix},
\qquad
\mathfrak m_2=
\begin{pmatrix} 0&1\\1&0\end{pmatrix},
\qquad
\mathfrak m_3=
\begin{pmatrix} 0&i\\-i&0\end{pmatrix},
\]
and the meromorphic $1$-forms $\omega_j$ are given by
\[
\begin{split}
\omega_1 &=
\left(
\tfrac{1}{z-p_1}-\tfrac{1}{z-p_2}
+\tfrac{1}{z-p_3}-\tfrac{1}{z-p_4}
\right)
=
\tfrac{4 i\sin(2\varphi)\,z}{z^4-2\cos(2\varphi)z^2+1}\,dz,\\
\omega_2 &=
\left(
\tfrac{1}{z-p_1}-\tfrac{1}{z-p_2}
-\tfrac{1}{z-p_3}+\tfrac{1}{z-p_4}
\right)
=
\tfrac{4\cos(\varphi)(z^2-1)}{z^4-2\cos(2\varphi)z^2+1}\,dz,\\
\omega_3 &=
\left(
\tfrac{1}{z-p_1}+\tfrac{1}{z-p_2}
-\tfrac{1}{z-p_3}-\tfrac{1}{z-p_4}
\right)
=
\tfrac{4i\sin(\varphi)(z^2+1)}{z^4-2\cos(2\varphi)z^2+1}\,dz.
\end{split}
\]

The parameters $x_1,x_2,x_3$ are meromorphic functions of $\lambda$ on a disk of
radius $\rho>1$ with simple poles at $\lambda=0$, i.e.,
\[
x_j(\lambda)=\sum_{k=-1}^{\infty}x_{j,k}\lambda^k.
\]
As in \cite{HHT2}, their residues at $\lambda=0$ are
\begin{equation}\label{eq:residue-x}
x_{1,-1}=\tfrac{1}{2},
\qquad
x_{2,-1}=-\tfrac{1}{2}\sin(\varphi),
\qquad
x_{3,-1}=-\tfrac{1}{2}\cos(\varphi).
\end{equation}
Up to scaling and signs, this is the unique choice ensuring that
$\Res_{\lambda=0}\eta$ is nilpotent.
In order to apply the implicit function theorem, we regard the parameters
$x_1,x_2,x_3$ as elements of the Wiener space $\mathcal W_{\mathbb{R},\rho}$ for
some fixed $\rho>1$; see Section~2.1 of \cite{HHT2}.
Equivalently, all coefficients $x_{j,k}$ are real and satisfy
\[
\sum_{k\in\mathbb{Z}} |x_{j,k}|\,\rho^{|k|}<\infty.
\]
Finally, the parameter $r$ is assumed to be a positive real number.

The potential has the following symmetries:
Let
$$
\delta(z)=-z
,\quad
\tau(z)=\frac{1}{z}
\quad\text{ and }\quad
\sigma(z)=\overline{z}.$$
Then
\begin{equation*}
\begin{split}
\delta^*\eta&=D^{-1}\eta D\quad\text{ with }\quad D=\begin{psmallmatrix}i&0\\0&-i\end{psmallmatrix}\\
\tau^*\eta&=C^{-1}\eta C\quad\text{ with }\quad C=\begin{psmallmatrix}0&i\\i&0\end{psmallmatrix}\\
\sigma^*\eta&=\overline{\eta}\quad\text{ where }\quad \overline{\eta}(z,\lambda)=\overline{\eta(z,\overline{\lambda})}.\\
\end{split}
\end{equation*}

\begin{Rem}\label{rem:compareHHT2}
The DPW potential used here is very close to the one in \cite{HHT2}.
In particular, the underlying $4$--punctured sphere and the symmetry assumptions
with respect to $\delta$, $\tau$, and $\sigma$ are identical.

The differences arise at the level of reconstructing and assembling the (complete) surface from the
potential.
First, the branched cover $\Sigma_k$ is defined using a different ordering of the
four punctures, leading to a slightly different realization of the symmetry
group.
Second, although the functions $x_j(\lambda)$ have the same analytic type and the
same residue data at $\lambda=0$ as in \cite{HHT2}, the extrinsic closing
conditions are adapted to the periodic setting by imposing translational
monodromy at the Sym point $\lambda=1$.
As a consequence, some of the $x_j$ acquire additional constant terms at the
initial value.
\end{Rem}

\subsubsection{Formulation of the monodromy problem}\label{sec:monodromyproblem}

In order to obtain a well-defined CMC immersion via the DPW method, one must
ensure that the monodromy of the associated family of flat connections is unitary for all
$\lambda\in S^{1}$.
We reformulate this requirement in the form of the problem
\eqref{eq:monodromy-problem}, which is amenable to the implicit function theorem
and allows the closing conditions to be solved systematically.  The analytic setup for this application of the implicit function theorem,
including the choice of function spaces, the linearization of the monodromy map,
and the verification of the required non-degeneracy conditions, is developed in
Appendix~\ref{section:IFT}.

Let $\Phi_{\mathrm{DPW}}\colon\widetilde\Sigma\to\Lambda SL(2,\mathbb{C})$ denote
the solution of
\[
d\Phi_{\mathrm{DPW}}=\Phi_{\mathrm{DPW}}\eta
\]
with initial condition $\Phi_{\mathrm{DPW}}(z=0)=\Id$.
Let $\gamma_1,\dots,\gamma_4$ be generators of the fundamental group
$\pi_1(\Sigma,0)$, where $\gamma_j$ encloses the singularity $p_j$.
Denote by $M_j=\mathcal M(\Phi_{\mathrm{DPW}},\gamma_j)$ the monodromy of
$\Phi_{\mathrm{DPW}}$ along $\gamma_j$.  Following the notation of \cite{HHT2}, set
\[
\mathcal P=\Phi_{\mathrm{DPW}}(z=1),
\qquad
\mathcal Q=\Phi_{\mathrm{DPW}}(z=i),
\]
and define
\[
L_1=\mathcal Q C D\mathcal Q^{-1},
\qquad
L_2=D,
\qquad
L_3=\mathcal P C\mathcal P^{-1},
\qquad
L_4=-\mathcal P C\mathcal P^{-1} D\mathcal Q C D\mathcal Q^{-1}.
\]
Then, by Equation~(12) in \cite{HHT2}, the monodromies satisfy
\[
M_1=L_4,\quad
M_2=L_1^{-1}L_4 L_1,\quad
M_3=L_2^{-1}L_4 L_2,\quad
M_4=L_2^{-1}L_1^{-1}L_4 L_1 L_2.
\]

The half-trace coordinates are defined by
\[
\mathfrak p=\tfrac12\tr(L_2L_3),
\qquad
\mathfrak q=-\tfrac12\tr(L_1L_2),
\qquad
\mathfrak r=-\tfrac12\tr(L_2L_4).
\]
A direct computation shows that
\begin{equation*}
\begin{split}
L_2L_3&=
\begin{pmatrix}
\mathcal P_{11}\mathcal P_{21}-\mathcal P_{12}\mathcal P_{22}
&
-\mathcal P_{11}^2+\mathcal P_{12}^2\\
-\mathcal P_{21}^2+\mathcal P_{22}^2
&
\mathcal P_{11}\mathcal P_{21}-\mathcal P_{12}\mathcal P_{22}
\end{pmatrix},\\[0.3em]
L_1L_2&=
-i\begin{pmatrix}
\mathcal Q_{11}\mathcal Q_{21}+\mathcal Q_{12}\mathcal Q_{22}
&
\mathcal Q_{11}^2+\mathcal Q_{12}^2\\
\mathcal Q_{21}^2+\mathcal Q_{22}^2
&
\mathcal Q_{11}\mathcal Q_{21}+\mathcal Q_{12}\mathcal Q_{22}
\end{pmatrix},
\end{split}
\end{equation*}
where $\mathcal P=(\mathcal P_{ij})$ and $\mathcal Q=(\mathcal Q_{ij})$.
Hence
\[
\mathfrak p=\mathcal P_{11}\mathcal P_{21}-\mathcal P_{12}\mathcal P_{22},
\qquad
\mathfrak q=i\bigl(\mathcal Q_{11}\mathcal Q_{21}+\mathcal Q_{12}\mathcal Q_{22}\bigr).
\]

Finally, define
\[
\mathcal K:=r^2(-x_1^2+x_2^2+x_3^2)
=
\det\!\left(\tfrac{1}{t}\Res_{p_j}(\eta)\right).
\]
We consider the following problem:
\begin{equation}\label{eq:monodromy-problem}
\begin{cases}
\mathfrak p(\lambda)\ \text{and}\ \mathfrak q(\lambda)\ \text{are real for all }\lambda\in\S^1,\\
\mathfrak q(1)=0,\\
\mathcal K(\lambda)=1\quad\text{for all }\lambda.
\end{cases}
\end{equation}
Since $\mathfrak p$, $\mathfrak q$, and $\mathfrak r$ satisfy a Fricke-type
algebraic relation, no additional constraint on $\mathfrak r$ is required.
Define the Fricke discriminant
\[
\Delta=4(1-\mathfrak p^2)(1-\mathfrak q^2)-4\cos^2(2\pi t).
\]

At the central value $t=0$, the monodromy representation degenerates to an abelian
one.
In this limit, the half--trace coordinates
$(\mathfrak p,\mathfrak q,\mathfrak r)$ and the quantity $\mathcal K$ admit
explicit expressions in terms of the residue data $x_j$.
These expressions are used to solve the monodromy problem for small $t$ via the
implicit function theorem; see Appendix~\ref{section:IFT}.
In Appendix~\ref{proof:prop:monodromy-problem} we prove
(Lemma~\ref{lem:monodromy-problem}) that the solution
$\Phi_{\mathrm{DPW}}$ has unitary monodromy up to conjugation provided that
Problem~\eqref{eq:monodromy-problem} is solved and that there exists
$\lambda_1\in\S^1$ such that $\Delta(\lambda_1)>0$.
In this case, $\mathfrak r$ is real and the monodromy representation lies in the
unitary component.

We can then apply Remark~\ref{rem:lgfac} to construct a CMC surface.
The geometric argument relating the monodromy data to the existence of a
translationally periodic CMC immersion is given in
Appendix~\ref{proof:prop:monodromy-problem}. 
\begin{Pro}
\label{prop:monodromy-problem}
Assume that Problem \eqref{eq:monodromy-problem} is solved, and
assume  that there exists $\lambda_1\in\S^1$ such that $\Delta(\lambda_1)>0$.
Then, 
the pullback $\tilde\eta:=\pi^*\eta$ to $\Sigma_k$ has only apparent singularities at $\pi^{-1}(p_1),\cdots,\pi^{-1}(p_4)$ and produces a conformal CMC immersion $f\colon\Sigma_k\to\R^3/\Gamma$ 
whose monodromy is translational. We denote its image by $\Lawson_{k,\varphi}$.
Furthermore, the area of $\Lawson_{k,\varphi}$ is equal to
$$\mathcal A(f)=8\pi\left(1-r \cos(\varphi)x_{2,0}+r \sin(\varphi) x_{3,0}\right)$$
\end{Pro}
The proof of the area formula follows the same strategy as in \cite{HHT2} and is
given in Appendix~\ref{section:IFT}, where the argument is adapted to the present
DPW potential.

\subsection{Area and volume formulas}
In this section, we apply Theorem \ref{theorem:computeK} to compute the enclosed volume of the surfaces $\Lawson_{k,\varphi}$ in terms of residues of the DPW potentials. 
Throughout this section, the parameters $x_j$ and $r$ depend smoothly on the
deformation parameter $t$ and are determined as the solution of the monodromy
problem \eqref{eq:monodromy-problem}; for notational simplicity, this dependence
is suppressed.

We first show the following lemma:

\begin{Lem}\label{lem:reducible}
For $\lambda=1$, the triple $(r x_1,r x_2,r x_3)$ satisfies for all small
deformation parameter $t$
\[
(r x_1,r x_2,r x_3)\big|_{\lambda=1}=(0,0,-1).
\]
\end{Lem}
\begin{proof}
By Lemma~\ref{lem:monodromy-problem} in
Appendix~\ref{proof:prop:monodromy-problem}, all monodromies $M_j$ commute at
$\lambda=1$.
Hence, their common eigenspaces define two line bundles over $\Sigma$ which are
parallel with respect to $\nabla=d+\eta$.
In particular, the connection $\nabla$ is fully reducible.

By Lemma~10 in \cite{Fuchsian}, it follows that $\eta$ is conjugate, by a constant
matrix, to a diagonal potential.
Equivalently, the residues
\[
A_j := \tfrac{1}{rt}\Res_{p_j}\eta, \qquad j=1,\dots,4,
\]
are simultaneously diagonalizable and therefore commute.

On the other hand, we compute
\[
[A_1,A_2]
=
\begin{pmatrix}
-4i x_2 x_3 & -4 x_1 x_3 \\
-4 x_1 x_3 & 4i x_2 x_3
\end{pmatrix},
\qquad
[A_1,A_4]
=
\begin{pmatrix}
4i x_2 x_3 & 4i x_1 x_2 \\
-4i x_1 x_2 & -4i x_2 x_3
\end{pmatrix}.
\]
Thus, at least two of the quantities $x_1,x_2,x_3$ must vanish.
Since $\mathcal K=1$, this implies
\[
(rx_1,rx_2,rx_3)\big|_{\lambda=1}
\in
\{(\pm1,0,0),(0,\pm1,0),(0,0,\pm1)\}.
\]
The claim now follows by continuity in the deformation parameter $t$, using the
values of $x_1,x_2,x_3$ and $r$ at $t=0$ given in
\eqref{eq:central-value}.
\end{proof}

Write $\lambda=e^{i\tau}$.
By Lemma~\ref{lem:reducible} and the symmetry assumptions imposed on the DPW
potential~$\eta$, we may expand $r x_1$, $r x_2$, and $r x_3$ at $\tau =0$ as
\begin{equation*}
\begin{split}
r x_1(e^{i\tau})&=\tau a_1+\tfrac{\tau^2}{2} a_2+O(\tau^3),\\
r x_2(e^{i\tau})&=\tau b_1+\tfrac{\tau^2}{2} b_2+O(\tau^3),\\
r x_3(e^{i\tau})&=-1+O(\tau^2).
\end{split}
\end{equation*}
Recall from Proposition~\ref{prop:monodromy-problem} that the area of a CMC
surface can be computed in terms of residues of its DPW potential.
The following theorem shows that the enclosed volume of the surfaces
$\Lawson_{k,\varphi}$ can likewise be computed from residue data of the DPW
potential.

\begin{The}\label{the:volume-lawson}
The enclosed volume of $\Lawson_{k,\varphi}$ is given by
\[
\mathcal V(f)
=
\tfrac{1}{3}\bigl(\mathcal A(f)-4\pi i\,(a_1 b_2-a_2 b_1)\bigr).
\]
\end{The}

\begin{proof}
Let $y$ be the coordinate of $\Sigma_k$ as in \eqref{eq:algeqyz}.
Consider the gauge
\[
g_0
=
\tfrac{1}{\sqrt{2i}}
\begin{pmatrix} i&-i\\1&1\end{pmatrix}
\begin{pmatrix}y^{1/2}&0\\0&y^{-1/2}\end{pmatrix}.
\]
Although $y^{1/2}$ is multivalued, the gauge action of $g_0$ is well defined on
$\Sigma_k$. A direct computation shows that $\tilde\eta=\pi^*\eta$ satisfies
\[
(d+\tilde\eta).g_0
=
d
+\tfrac{1}{2k}r x_1\,\pi^*\omega_1
\begin{pmatrix}0&-i y^{-1}\\-i y&0\end{pmatrix}
+\tfrac{1}{2k}r x_2\,\pi^*\omega_2
\begin{pmatrix}0&-i y^{-1}\\i y&0\end{pmatrix}
+\tfrac{1}{2k}(r x_3+1)\,\pi^*\omega_3
\begin{pmatrix}1&0\\0&-1\end{pmatrix}.
\]

By Lemma~\ref{lem:reducible}, we have $(d+\tilde\eta).g_0=d$ at $\lambda=1$.
Thus we may expand $\eta$, defined by $d+\eta=(d+\tilde\eta).g_0$, as in
\eqref{eq:eta-exp}, with
\[
\newPhi_1
=
\tfrac{1}{2k}a_1\,\omega_1
\begin{pmatrix}0&-i y^{-1}\\-i y&0\end{pmatrix}
+\tfrac{1}{2k}b_1\,\omega_2
\begin{pmatrix}0&-i y^{-1}\\i y&0\end{pmatrix},
\]
and
\[
\newPhi_2
=
\tfrac{1}{2k}a_2\,\omega_1
\begin{pmatrix}\ast&-i y^{-1}\\-i y&\ast\end{pmatrix}
+\tfrac{1}{2k}b_2\,\omega_2
\begin{pmatrix}\ast&-i y^{-1}\\i y&\ast\end{pmatrix}.
\]
Here the starred entries are irrelevant for the residue computation, since
$\newPhi_1$ and $F_1^j$ are off-diagonal.

Note that $\newPhi_1$ is off-diagonal, so Theorem~\ref{theorem:computeK} applies.
In a neighborhood of $p_1$, we use $y$ as a local coordinate.
We have
\[
\tfrac{1}{2k}\omega_1=\frac{1}{2y}\,dy+O(y^{k-1})\,dy,
\qquad
\tfrac{1}{2k}\omega_2=\frac{1}{2y}\,dy+O(y^{k-1})\,dy,
\]
and hence
\[
\newPhi_1
=
\begin{pmatrix}
0&(-i a_1-i b_1)y^{-2}\\
- i a_1+i b_1&0
\end{pmatrix}
\tfrac{dy}{2}
+O(y^{k-2})\,dy.
\]
A primitive of $\newPhi_1$ is given by
\[
F_1^1
=
\tfrac{1}{2}
\begin{pmatrix}
0&(i a_1+i b_1)y^{-1}\\
(-i a_1+i b_1)y&0
\end{pmatrix}
+O(y^{k-1}),
\]
and similarly
\[
\newPhi_2
=
\begin{pmatrix}
\ast&(-i a_2-i b_2)y^{-2}\\
- i a_2+i b_2&\ast
\end{pmatrix}
\tfrac{dy}{2}
+O(y^{k-2})\,dy.
\]
Therefore, at $p_1$ we obtain
\[
\Res_{p_1}\tr\bigl(F_1^1\newPhi_2\bigr)
=
\tfrac{1}{2}(a_2 b_1-a_1 b_2).
\]

An analogous computation at $p_2,p_3,p_4$ yields
\[
\Res_{p_j}\tr\bigl(F_1^j\newPhi_2\bigr)
=
\tfrac{1}{2}(a_2 b_1-a_1 b_2),
\qquad j=2,3,4.
\]
Collecting all contributions, Theorem~\ref{theorem:computeK} gives
\[
K=2(a_1 b_2-a_2 b_1).
\]
Substituting this expression for $K$ into
Theorem~\ref{Thm:enclosedvolr3} yields the stated formula for $\mathcal V(f)$.
\end{proof}

\section{A counterexample to the isoperimetric problem in $\mathbb{T}^2\times\mathbb{R}$}
\label{sec:iso}
\begin{figure}[b]
\includegraphics[width=6.5cm]{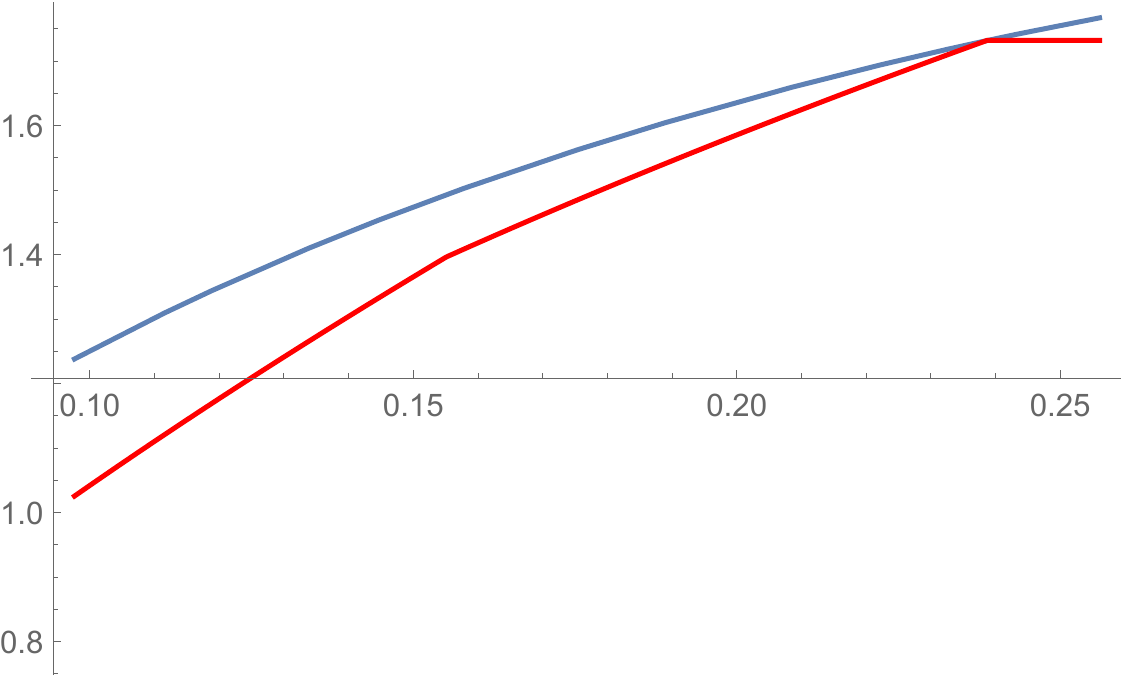}
\hspace{1cm}
\includegraphics[width=6.5cm]{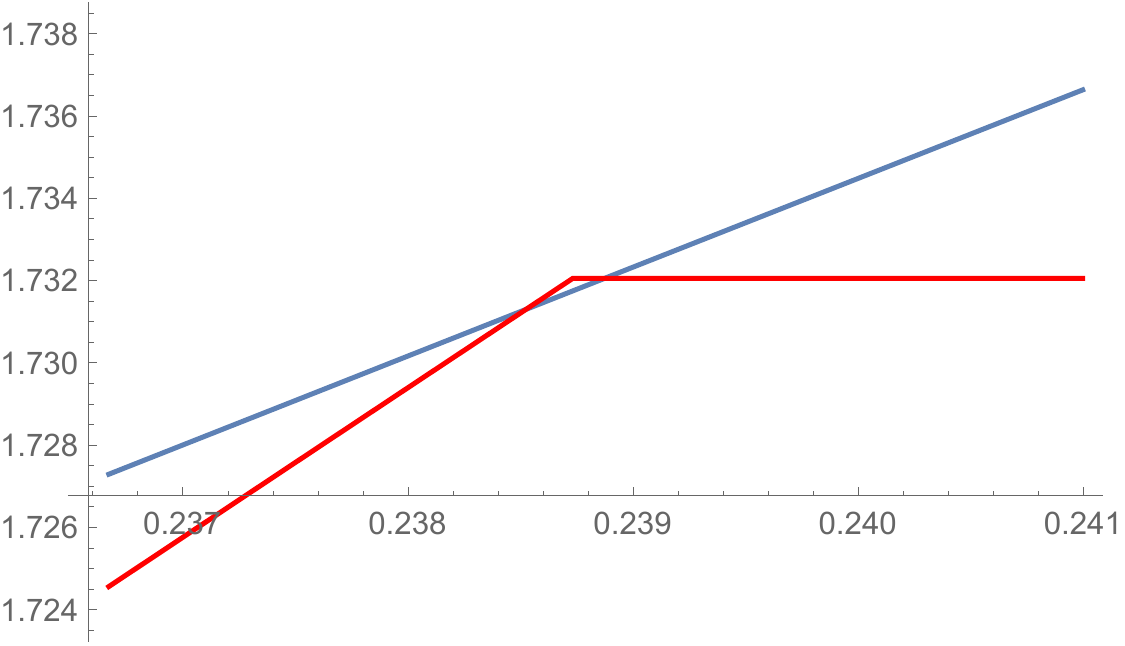}
\caption{
Comparison of isoperimetric profiles in $\mathbb{T}^2\times\mathbb{R}$ for the
equilateral torus.
The figure plots the area as a function of enclosed volume for the (rescaled) Lawson-type
family $\Lawson_{3,\varphi}$ (blue curves) and for the conjectured isoperimetric
minimizers -- spheres, cylinders, and pairs of horizontal planes (red curves).
Near the critical volume $\tfrac{3}{4\pi}$, corresponding to the transition from
cylinders to planes, the Lawson-type family attains  smaller area than
the conjectured competitors, providing a counterexample to the
isoperimetric conjecture of \cite{HPRR}.
}\label{fig:isoperimetric1}
\end{figure}

Let $\mathbb{T}^2$ be a flat $2$--torus.
Hauswirth, P\'erez, Romon, and Ros \cite{HPRR} proposed the following conjecture.

\begin{Con}\label{conj:HPRR}
The only solutions of the isoperimetric problem in
$\mathbb{T}^2\times\mathbb{R}$ are spheres, cylinders, and pairs of horizontal
planes.
\end{Con}

In the following, we normalize all period lattices so that the shortest vectors have length 1. All areas and volumes are computed with respect to this normalization.

 We now exhibit a counterexample to this conjecture in the case of the
equilateral torus
\[
\mathbb{T}^2=\mathbb{C}/\langle 1,e^{i\pi/3}\rangle.
\]
Specifically, we first consider numerical examples of the Lawson-type genus-2 surfaces
\(\Lawson_{3,\varphi}\) and rescale them accordingly. 
The resulting period lattice is the normalized equilateral lattice
$\Gamma=\langle 1,e^{i\pi/3}\rangle$. 
Later, using a sufficiently fine triangulation of the resulting surface, this will yield a genuine counterexample.

Figure~\ref{fig:isoperimetric1} compares the isoperimetric profiles of the (rescaled)
$\Lawson_{3,\varphi}$ family with those of the conjectured minimizers (spheres,
cylinders, and pairs of planes). 
The figure shows that the Lawson-type family $\Lawson_{3,\varphi}$ achieves strictly lower area than
the conjectured competitors for the same enclosed volume, in a neighborhood of
the critical volume $\tfrac{3}{4\pi}$.
This value corresponds to the transition from the cylinder to a pair of planes.
More precisely, using our Minkowski formula from Theorem \ref{Thm:enclosedvolr3} combined with Theorem \ref{theorem:computeK},
for the surface $\Lawson_{3,\varphi}$ with $\varphi=1.454838491$ shown in Figure \ref{fig:lawson4-views}, we obtain the normalized values
\begin{equation}\label{eq:theodpwva}
\mathcal V \simeq 0.2387324146 \simeq \tfrac{3}{4\pi}
\qquad\text{and}\qquad
\mathcal A \simeq 1.731745356 < \sqrt{3} \simeq 1.73205081,
\end{equation}
where $\sqrt{3}$ is the area of the conjectured competitors in $\mathbb{T}^2\times\mathbb{R}$ of the same enclosed volume
$ \tfrac{3}{4 \pi}.$


\subsubsection*{Square period lattice}
It is natural to ask the same question for the square torus
$\mathbb{T}^2=\mathbb{C}/\mathbb{Z}^2$.
In this case, a natural competing family is given by $\Lawson_{4,\varphi}$, which have a
square period lattice.
These surfaces have genus~$3$; however, after quotienting by the full period
lattice, the resulting surface has genus~$2$.
As before, we rescale the surfaces so that their shortest period has length~$1$.

Figure~\ref{fig:isoperimetric2} compares the isoperimetric profiles of the rescaled
$\Lawson_{4,\varphi}$ family with those of the conjectured minimizers, namely
spheres, cylinders, and pairs of horizontal planes.
In contrast to the equilateral case, the Lawson-type family
$\Lawson_{4,\varphi}$ does not attain smaller area than the conjectured
minimizers for the same enclosed volume.
This suggests that the conjecture may still hold for the square torus.

At the critical volume $\mathcal V=\tfrac{1}{\pi}$, corresponding to the
transition from cylinders to pairs of planes, we find that the rescaled Lawson-type
surface $\Lawson_{4,\varphi}$ has area
\[
\mathcal A \simeq 2.0353
\]
This value agrees with that computed by Ros \cite{Ros} using Brakke's Surface
Evolver, up to normalization of the period lattice.
Indeed, Ros's convention considers only a fundamental half-domain, and the area
reported in \cite{Ros} is $1.017$, which is half of the value
2.0353 obtained here.

\begin{figure*}[t]
\centering
\begin{minipage}[t]{0.46\textwidth}
\centering
\includegraphics[width=\textwidth]{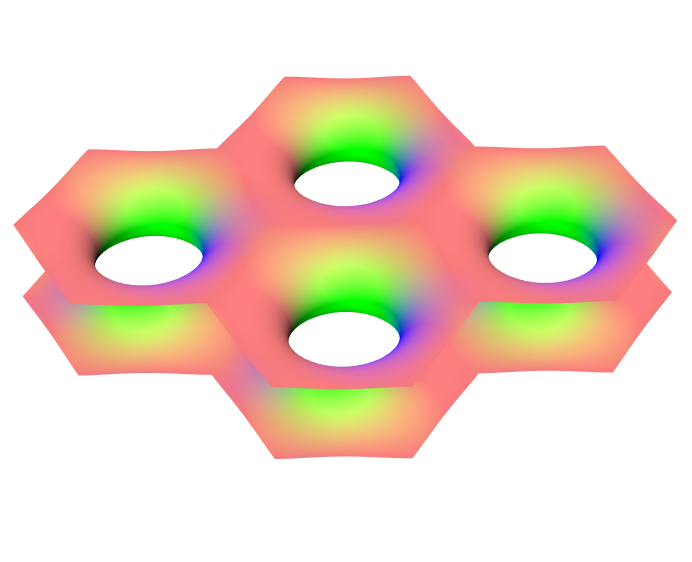}
\end{minipage}\hfill
\begin{minipage}[t]{0.48\textwidth}
\centering
\includegraphics[width=\textwidth]{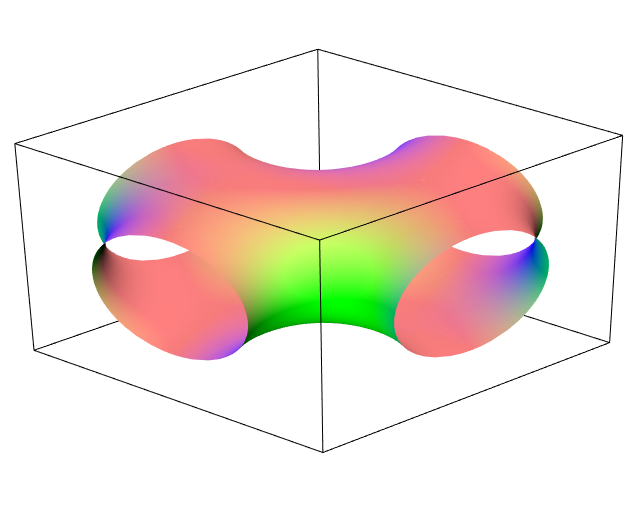}
\end{minipage}
\caption{
Two views of the Lawson-type CMC surface $\Lawson_{3,\varphi}$ at
$\varphi=1.454838491$.
(a) Periodic extension of the surface in $\mathbb{R}^3$.
(b) The same surface restricted to a fundamental domain of the equilateral
lattice together with a finite height in the $\mathbb{R}$-direction. Images by the third author.
}
\label{fig:lawson4-views}
\end{figure*}

\subsection{Numerical aspects}

The Monodromy Problem is solved theoretically for large $k$ in
Appendix~\ref{section:IFT}.
In order to study the families $\Lawson_{3,\varphi}$ and
$\Lawson_{4,\varphi}$, we solve the Monodromy Problem numerically.
In the test case $\Lawson_{2,\varphi}$, i.e. round cylinders with varying radii,  for which both the area and the enclosed
volume are known explicitly, the numerical results agree with the exact values
to high precision, providing confidence in the reliability of our numerical
scheme.

We briefly summarize the numerical procedure.
\begin{itemize}
\item
Functions of the spectral parameter $\lambda$ are expanded into Fourier series
and truncated at some finite order $n$ (typically $n=10$, or $n=20$ for higher
accuracy).

\item
The matrices $\mathcal P$ and $\mathcal Q$ are computed by numerically solving
the ODE
\[
d\Phi_{\mathrm{DPW}}=\Phi_{\mathrm{DPW}}\eta
\]
using a fourth--order Runge--Kutta method.
Typically, $100$ subdivisions are sufficient in practice, while $500$
subdivisions are used for more accurate computations.

\item
The Monodromy Problem is solved using Newton's method.
For $\varphi$ near $\tfrac{\pi}{4}$, we use the central value
\eqref{eq:central-value} as an initial guess.
Typically, Newton's method converges in $5$--$8$ iterations.
For more extreme values of $\varphi$, previously computed solutions are used as
initial values.

\item
We observe that the solutions $(x_1,x_2,x_3,r)$ satisfy
$|x_{j,m}|=O(10^{-m})$, which a posteriori justifies the chosen truncation order
in $\lambda$.

\item
The unitarizing gauge, the period lattice, and the geometric quantities (area
and enclosed volume) are computed numerically using
Proposition~\ref{prop:monodromy-problem} and
Theorem~\ref{the:volume-lawson}.

\item
As a validation test, consider the surface
$\Lawson_{2,\tfrac{\pi}{4}}$, the conjugate cousin of the Clifford torus
$\xi_{1,1}$.
This surface is the flat CMC cylinder of height $\tfrac{h}{2}=\pi$, with exact
area $\pi^2$ and enclosed volume $\tfrac{1}{4}\pi^2$.
For $n=20$ and using Runge--Kutta with $500$ subdivisions, the numerically
computed values for the area and enclosed volume differ from the exact values by
less than $2.5\times10^{-14}$ and $3.2\times10^{-11}$, respectively.

\item
Since the initial value \eqref{eq:central-value} corresponds to the limit
$t\to0$, the numerical scheme performs particularly well for larger values of
$k$.
We therefore expect at least comparable accuracy for $k=3,4$ as in the test case
$k=2$.

\item
All computations were carried out using \emph{Mathematica}.
The corresponding notebook is attached to the {\em arXiv} submission.
\item All triangulations and images were constructed with the potential from Section \ref{sec:dpwlawsonpot} using a C++ implementation of the DPW method.\end{itemize}

\begin{figure}[t]
\centering
\includegraphics[width=6.5cm]{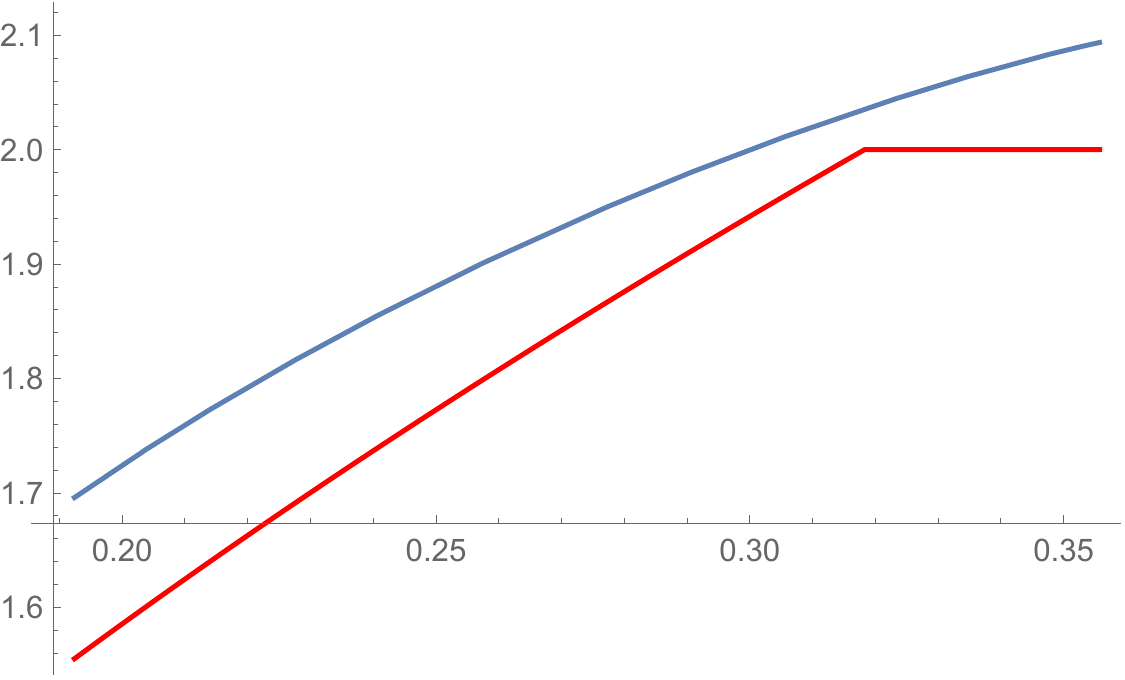}
\caption{
Red: isoperimetric profile $\mathcal V\mapsto\mathcal A$ of the conjectured
minimizers (sphere, cylinder, and pair of planes) in
$\mathbb{T}^2\times\mathbb{R}$, where $\mathbb{T}^2$ is the square torus.
Blue: isoperimetric profile of the $\Lawson_{4,\varphi}$ family for
$\varphi\in[\tfrac{\pi}{6},\tfrac{9\pi}{20}]$.
}
\label{fig:isoperimetric2}
\end{figure}

\subsection{The triangulated counterexample for the hexagonal lattice}\label{sec:ce}
We consider the CMC surface $\Lawson_{3,\varphi}$ with $\varphi = 1.454838491$ 
in $\T^2\times\R$, where $\mathbb{T}^2=\mathbb{C}/\langle 1,e^{i\pi/3}\rangle$; see Figure \ref{fig:lawson4-views}. Applying the DPW method to the numerical DPW data yields an explicit triangulated surface in $\mathbb T^2\times\mathbb R$.
The enclosed volume of the triangulated surface is computed by summing the volumes of
vertical prisms, where each prism has as its top face one triangle of the surface
and as its base the horizontal projection of that triangle onto the reflection
plane of the surface.
The surface area is computed by summing the areas of all triangles in the
triangulation.

The resulting values are in good agreement\footnote{ It is expected that a triangulated approximation of a stable CMC surface has slightly smaller enclosed volume and slightly larger area.} with the values  \eqref{eq:theodpwva}
for our numerical DPW data.
Using a triangulation with $2236$ triangles (for a fundamental 4-gon) we obtain
\[
\mathcal A \simeq 1.731823,
\qquad
\mathcal V \simeq 0.238698.
\]
This triangulated surface has strictly smaller normalized area  than the standard cylindrical competitor 
(of area 1.73193) for the same enclosed volume, by approximately $1.2\times 10^{-4},$ and therefore provides a counterexample to Conjecture~\ref{conj:HPRR}.

\subsubsection*{Quantitative comparison with the standard competitors}
We record more precise numerical information on the range of parameters $\varphi$ for which the Lawson-type surfaces \(\Lawson_{3,\varphi}\) improve upon the standard competitors (spheres, cylinders, and pairs of planes) in the isoperimetric problem in $\mathbb{T}^2 \times \mathbb{R},$ where $\mathbb{T}^2 = \mathbb{C}/\langle 1, e^{i\pi/3} \rangle$ is as above. This occurs for
\[
\varphi \in [\,1.454356,\;1.455165\,].
\]
Here the interval bounds are given with rounding; with this convention, the Lawson-type surfaces
$\Lawson_{3,\varphi}$
strictly outperform the standard competitors  for normalized volumes throughout the closed interval
\[
\mathcal V \in [\,0.238524,\;0.238873\,].
\]



\appendix
\section{Solution of the Monodromy Problem}\label{section:IFT}

This appendix provides the analytic and geometric input needed to solve the
monodromy problem formulated in Section~\ref{sec:monodromyproblem}.
Subsection~\ref{section:IFT1} establishes solvability via the implicit function theorem,
Subsection~\ref{app:monodromy-solvability} analyzes the induced symmetries, and Subsection~\ref{proof:prop:monodromy-problem} relates the
monodromy data to the existence of periodic CMC immersions and gives a proof of Proposition \ref{prop:monodromy-problem}.

\subsection{Functional--analytic setup and linearization}
\label{section:IFT1}

\begin{Pro}
For $t$ sufficiently small, there exists a unique solution
\[
(x_1(t),x_2(t),x_3(t),r(t))
\in(\mathcal W_{\mathbb{R},\rho})^3\times\mathbb{R}
\]
of Problem~\eqref{eq:monodromy-problem}, depending analytically on $t$, and such
that $\Delta(i)>0$.
Moreover, at $t=0$ the solution is given by
\begin{equation}\label{eq:central-value}
x_1(0)=\tfrac{1-\lambda^2}{2\lambda},\quad
x_2(0)=-\sin(\varphi)\tfrac{(\lambda-1)^2}{2\lambda},\quad
x_3(0)=-\tfrac{1}{\cos(\varphi)}-\cos(\varphi)\tfrac{(\lambda-1)^2}{2\lambda},
\quad\text{and}\quad
r(0)=\cos(\varphi).
\end{equation}
\end{Pro}

\begin{proof}
At $t=0$ the potential vanishes, $\eta=0$, and hence
$\Phi_{\mathrm{DPW}}=\Id$.
It follows that $\mathfrak p=\mathfrak q=\mathfrak r=0$.
Therefore the rescaled quantities
\[
\widehat{\mathfrak p}:=\tfrac{1}{t}\mathfrak p,\qquad
\widehat{\mathfrak q}:=\tfrac{1}{t}\mathfrak q,\qquad
\widehat{\mathfrak r}:=\tfrac{1}{t}\mathfrak r
\]
extend analytically to $t=0$.

With this notation, Problem~\eqref{eq:monodromy-problem} is equivalent to the
system
\begin{equation}\label{eq:monodromy-problem2}
\begin{cases}
\mathcal F_1(t,r,x):=\widehat{\mathfrak p}-\widehat{\mathfrak p}^{\,*}=0,\\
\mathcal F_2(t,r,x):=\widehat{\mathfrak q}-\widehat{\mathfrak q}^{\,*}=0,\\
\mathcal H(t,r,x):=\widehat{\mathfrak q}(\lambda=1)=0,\\
\mathcal K(r,x)=1,
\end{cases}
\end{equation}
where the involution $*$ on $\mathcal W$ is defined by
\[
u^*(\lambda)=\overline{u(1/\overline{\lambda})}.
\]

We first solve Problem~\eqref{eq:monodromy-problem2} at $t=0$.
There the system admits several solutions; however, only one of
them satisfies the positivity condition $\Delta(i)>0$.
By Proposition~\ref{prop:monodromy-problem}, for $t>0$ a solution of
Problem~\eqref{eq:monodromy-problem2} with $\Delta(i)>0$ yields a unitarizable
monodromy representation $(L_1,\dots,L_4)$.
In particular, $\mathfrak r$ is real on the unit circle, which implies the
additional constraint
\begin{equation}\label{eq:monodromy-complementary}
\mathcal F_3(t,r,x):=\widehat{\mathfrak r}-\widehat{\mathfrak r}^{\,*}=0.
\end{equation}
Thus, in order to select the relevant branch of solutions, we solve
Problem~\eqref{eq:monodromy-problem2} together with
\eqref{eq:monodromy-complementary} at $t=0$.

Assume that $t=0$.
By Proposition~12 in \cite{HHT2} and Proposition~57 in \cite{HHTSD}, we have
\[
\widehat{\mathfrak p}\big|_{t=0}=2\pi r x_3,\qquad
\widehat{\mathfrak q}\big|_{t=0}=2\pi r x_2,
\qquad\text{and}\qquad
\widehat{\mathfrak r}\big|_{t=0}=2\pi i r x_1.
\]
Recall that the negative parts of $x_1,x_2,x_3$ are fixed.
The equations
$\mathcal F_1=\mathcal F_2=\mathcal F_3=0$ and $\mathcal H=0$
then imply that $x_1,x_2,x_3$ are Laurent polynomials of degree~$1$, given by
\[
\begin{cases}
x_1=\tfrac{1}{2}(\lambda^{-1}-\lambda),\\[0.2em]
x_2=-\tfrac{1}{2}\sin(\varphi)(\lambda^{-1}-2+\lambda),\\[0.2em]
x_3=-\tfrac{1}{2}\cos(\varphi)(\lambda^{-1}+\lambda)+x_{3,0}.
\end{cases}
\]

Substituting these expressions into $\mathcal K$ yields
\[
\mathcal K
=
r^2\Bigl[
(\lambda^{-1}+\lambda)\bigl(-\sin^2(\varphi)-\cos(\varphi)x_{3,0}\bigr)
+
\lambda^0\Bigl(
\tfrac{1}{2}+\tfrac{3}{2}\sin^2(\varphi)+\tfrac{1}{2}\cos^2(\varphi)+x_{3,0}^2
\Bigr)
\Bigr].
\]
Vanishing of the coefficients of $\lambda^{\pm1}$ determines
\[
x_{3,0}
=
-\frac{\sin^2(\varphi)}{\cos(\varphi)}
=
-\frac{1}{\cos(\varphi)}+\cos(\varphi),
\]
and $\mathcal K$ simplifies to
\[
\mathcal K=\frac{r^2}{\cos^2(\varphi)},
\]
which yields $r=\cos(\varphi)$.

Following the proof of Proposition~16 in \cite{HHT2}, we now solve the system
\eqref{eq:monodromy-problem2} for small $t$ using the Implicit Function Theorem.
By definition, $\mathcal F_j=-\mathcal F_j^*$, and by symmetry
(Proposition~14 in \cite{HHT2}) we have $\overline{\mathcal F}_j=\mathcal F_j$.
Thus it suffices to solve $\mathcal F_j^+=0$.

The partial differential of $(\mathcal F_1^+,\mathcal F_2^+,\mathcal H)$ with
respect to the parameters $(x_1,x_2,x_3,r)$ at $t=0$ is
\begin{align*}
d\mathcal F_1^+&=2\pi\cos(\varphi)\,dx_3^+,\\
d\mathcal F_2^+&=2\pi\cos(\varphi)\,dx_2^+,\\
d\mathcal H&=2\pi\cos(\varphi)\,dx_3(\lambda=1)
=2\pi\cos(\varphi)\bigl(dx_3^0+dx_3^+(\lambda=1)\bigr).
\end{align*}
The partial differential of $(\mathcal F_1^+,\mathcal F_2^+,\mathcal H)$ with
respect to $(x_2,x_3^+)$ is therefore an isomorphism from
$\mathcal W^{\ge0}_{\mathbb R}\times\mathcal W^{>0}_{\mathbb R}$ to
$\mathcal W^{>0}_{\mathbb R}\times\mathcal W^{>0}_{\mathbb R}\times\mathbb R$.
By the Implicit Function Theorem, the equations
$(\mathcal F_1^+,\mathcal F_2^+,\mathcal H)=(0,0,0)$ uniquely determine
$(x_2,x_3^+)$ as functions of $t$ and the remaining parameters
$(x_1,x_3^0,r)$.
Moreover, the partial derivatives of $x_2$ and $x_3^+$ with respect to these
remaining parameters vanish at $t=0$.

The partial differential of $\mathcal K$ with respect to the remaining parameters
is
\begin{align*}
d\mathcal K
&=2r\,dr+r^2(-2x_1\,dx_1+2x_3\,dx_3^0)\\
&=2\cos(\varphi)\,dr
+\cos^2(\varphi)(\lambda-\lambda^{-1})\,dx_1
+\bigl(-2\cos(\varphi)-\cos^3(\varphi)(\lambda-2+\lambda^{-1})\bigr)\,dx_3^0.
\end{align*}
Using Proposition~66 in Appendix~A of \cite{HHT2}, we expand $\lambda\mathcal K$ as
\[
\lambda\mathcal K
=
\mathcal K_{-1}+\lambda\mathcal K_0+(\lambda^2-1)\widehat{\mathcal K},
\qquad
\widehat{\mathcal K}\in\mathcal W^{\ge0}_{\mathbb R}.
\]
This yields
\begin{align*}
d\widehat{\mathcal K}&=\cos^2(\varphi)\,dx_1-\cos^3(\varphi)\,dx_3^0,\\
d\mathcal K_0&=2\cos(\varphi)\,dr-2\cos(\varphi)\sin^2(\varphi)\,dx_3^0,\\
d\mathcal K_{-1}&=-2\cos^3(\varphi)\,dx_3^0.
\end{align*}
The partial differential of $(\widehat{\mathcal K},\mathcal K_0,\mathcal K_{-1})$
with respect to $(x_1,r,x_3^0)$ has upper triangular form and defines an
automorphism of $\mathcal W^{\ge0}_{\mathbb R}\times\mathbb R^2$.
By the Implicit Function Theorem, the equation
$(\widehat{\mathcal K},\mathcal K_0,\mathcal K_{-1})=(0,1,0)$ uniquely determines
$(x_1,r,x_3^0)$ as analytic functions of $t$ for $t$ sufficiently small.

Finally, we compute
\[
\Delta
=
4\bigl(1-t^2\widehat{\mathfrak p}^2\bigr)
\bigl(1-t^2\widehat{\mathfrak q}^2\bigr)
-4\cos^2(2\pi t)
=
16\pi^2 t^2\bigl(1-r^2x_2^2-r^2x_3^2\bigr)+O(t^4)
=
-16\pi^2 t^2 r^2x_1^2+O(t^4).
\]
At $t=0$ and $\lambda=i$ we have $x_1=-i$, and therefore
$\Delta(\lambda=i)>0$ for $t>0$ sufficiently small.
\end{proof}

\subsection{The proof of Proposition~\ref{prop:monodromy-problem}}
\label{proof:prop:monodromy-problem}

We first prove the following lemma.

\begin{Lem}\label{lem:monodromy-problem}
Assume that Problem~\eqref{eq:monodromy-problem} is solved and that there exists
$\lambda_1\in\S^1$ such that $\Delta(\lambda_1)>0$.
Then:
\begin{enumerate}
\item
There exists a diagonal unitarizer $U\in\Lambda SL(2,\mathbb{C})$ such that
$U L_j U^{-1}\in\Lambda SU(2)$ for all $j$, and hence also
$U M_j U^{-1}\in\Lambda SU(2)$ for all $j$.
\item
At $\lambda=1$, we have
\[
\mathfrak p(1)=\pm\sin(2\pi t),\qquad
\mathfrak r(1)=0,\qquad
M_1(1)=M_2(1)=M_3(1)^{-1}=M_4(1)^{-1}.
\]
\end{enumerate}
\end{Lem}

\begin{proof}
Point~(1) is Proposition~11 in \cite{HHT2}.
That proposition assumes the existence of $\lambda_1\in\S^1$ such that
$\mathfrak p(\lambda_1)=\mathfrak q(\lambda_1)=0$, but this hypothesis is used
only to ensure $\Delta(\lambda_1)>0$, which we assume here.

\medskip

For point~(2), the $\sigma$ and $\delta$ symmetries imply, as in \cite{HHT2}, that
$\overline{\mathcal P}=\mathcal P$ and $\overline{\mathcal Q}=D^{-1}\mathcal QD$.
Hence
\begin{align*}
\tr(\overline{L_1L_3})
&=\tr\!\left((D^{-1}\mathcal QD)(CD)(D^{-1}\mathcal Q^{-1}D)\,
\mathcal P(-C)\mathcal P^{-1}\right)\\
&=\tr\!\left(-D\mathcal P C\mathcal P^{-1}D\,\mathcal QCD\mathcal Q^{-1}\right)
=\tr(L_2L_4).
\end{align*}
On the other hand, if $\mathfrak q=0$, then $L_1L_2$ is off-diagonal, hence $L_1$
is off-diagonal and $L_2L_1L_2=L_1$.
Therefore, at $\lambda=1$,
\[
\tr(L_2L_4)=-\tr(L_3DL_1D)=-\tr(L_3L_1).
\]
Since $\tr(L_2L_4)=-2\mathfrak r$ is real at $\lambda=1$, we conclude that
$\mathfrak r(1)=0$.
The value of $\mathfrak p(1)$ then follows from Fricke equation~(14) in
\cite{HHT2}.

Using $L_1^2=L_2^2=L_3^2=-\Id$ and $L_4=-L_3L_2L_1$, we compute
\[
\tr(M_1M_3)
=
-\tr((L_2L_4)^2)
=
2-(\tr(L_2L_4))^2
=
2-4\mathfrak r^2,
\]
and
\[
\tr(M_2M_3)
=
\tr\!\bigl(L_1(L_3L_2L_1)L_1L_2(L_3L_2L_1)L_2\bigr)
=
-\tr((L_1L_2)^2)
=
2-4\mathfrak q^2.
\]
Hence, at $\lambda=1$ we have $\tr(M_1M_3)=2$ and $UM_1M_3U^{-1}\in SU(2)$, so
$M_1M_3=\Id$.
Similarly, $M_2M_3=\Id$.
\end{proof}

\begin{proof}[Proof of Proposition~\ref{prop:monodromy-problem}]
Let $\widetilde\eta=\pi^*\eta$ be the pullback of $\eta$ to $\Sigma_k$, with poles
at $\widetilde p_j=\pi^{-1}(p_j)$.
Let $\widetilde\Phi_{\mathrm{DPW}}$ denote the solution of
$d\widetilde\Phi_{\mathrm{DPW}}=\widetilde\Phi_{\mathrm{DPW}}\widetilde\eta$ on
\[
\Sigma_k^*=\Sigma_k\setminus\{\widetilde p_1,\dots,\widetilde p_4\},
\]
with initial condition $\widetilde\Phi_{\mathrm{DPW}}(\widetilde 0)=U$, where
$\pi(\widetilde 0)=0$.
For $\gamma\in\pi_1(\Sigma_k^*,\widetilde 0)$, we can write
\[
\pi\circ\gamma=\prod_{j=1}^{\ell}\gamma_{i_j}
\]
for some $\ell\in\mathbb{N}$ and indices $i_1,\dots,i_\ell\in\{1,2,3,4\}$.
The condition that $\gamma$ is closed translates to
\[
\sum_{j=1}^{\ell}\varepsilon(i_j)\equiv 0 \pmod{k},
\]
where $\varepsilon(1)=\varepsilon(2)=1$ and $\varepsilon(3)=\varepsilon(4)=-1$.
The monodromy of $\widetilde\Phi_{\mathrm{DPW}}$ along $\gamma$ is
\[
\mathcal M(\widetilde\Phi_{\mathrm{DPW}},\gamma)
=
\prod_{j=1}^{\ell}U M_{i_j}U^{-1}
\in\Lambda SU(2).
\]
Moreover, by point~(2), at $\lambda=1$ we have
\[
\mathcal M(\widetilde\Phi_{\mathrm{DPW}},\gamma)\big|_{\lambda=1}
=
UM_1(1)^{\sum_{j=1}^{\ell}\varepsilon(i_j)}U^{-1}
=
\pm\Id,
\]
since $M_1^k=-\Id$.
Hence the immersion $f$ has translational periods only.

In a neighborhood of $\widetilde p_1$ we may use $y$ as a local coordinate, and
we have
\[
\widetilde\eta
=
krt\,A_1\,\frac{dy}{y}+O(y^{k-1})\,dy
=
\frac{r}{2}A_1\,\frac{dy}{y}+O(y^{k-1})\,dy.
\]
As in \cite{HHT2}, consider the gauge
\[
G_1=
\begin{pmatrix}1&0\\ \kappa&1\end{pmatrix}
\begin{pmatrix} y^{-1/2}&0\\0&y^{1/2}\end{pmatrix},
\qquad
\kappa=\frac{1-iry_1}{ry_2+iry_3}.
\]
Then
\[
\widetilde\eta.G_1
=
\begin{pmatrix}
0 & r(y_2+iy_3)\\
\frac{\mathcal K-1}{r(y_2+iy_3)y^2} & 0
\end{pmatrix}
\frac{dy}{2y}
+O(y^{k-2}),
\]
which is holomorphic at $y=0$ since $\mathcal K=1$.
Thus $\widetilde p_1$ is an apparent singularity; the same holds for
$\widetilde p_2,\widetilde p_3,\widetilde p_4$ by symmetry (or using analogous
gauges).

Finally, the proof of the area formula is identical to the proof of point~(2) of
Proposition~19 in \cite{HHT2}, based on formula~(22) in \cite{HHT2} for the area
in terms of regularizing gauges, and we omit the details.
\end{proof}

\subsection{Symmetries and geometric constraints}
\label{app:monodromy-solvability}
In this section, we prove that the immersion $f$ has the desired symmetries by
analyzing the image of the boundary of the quarter disk
\[
|z|<1,\qquad \Re(z)>0,\qquad \Im(z)>0.
\]
We use the standard identification of $\mathfrak{su}(2)$ with $\mathbb{R}^3$,
given by
\[
(y_1,y_2,y_3)\;\longleftrightarrow\;
\begin{pmatrix}
-i y_3 & y_1+i y_2\\
-y_1+i y_2 & i y_3
\end{pmatrix}.
\]
(In this subsection, $y_1,y_2,y_3$ denote the coordinates in $\mathbb{R}^3$ to
avoid confusion with the parameters $x_1,x_2,x_3$.
With respect to the geometric description given at the beginning of
Section~\ref{section:lawson}, we regard $y_1$ as the vertical direction.)

\begin{Pro}
The images of the segments $[0,1]$ and $[0,i]$ are symmetry curves lying in the
planes $y_1=0$ and $y_2=0$, respectively.
The images of the circular arcs $\widefrown{[p_1,i]}$ and
$\widefrown{[0,p_1]}$ are symmetry curves lying in planes parallel to $y_3=0$
and to $\sin(2\pi t)y_2=\cos(2\pi t)y_3$, respectively, where
$t=\tfrac{1}{2k}$.
\end{Pro}

\begin{proof}
We follow closely the proof of Point~(3) of Proposition~19 in \cite{HHT2}.
Throughout this subsection, $\Phi_{\mathrm{DPW}}$ denotes the solution of
\[
d\Phi_{\mathrm{DPW}}=\Phi_{\mathrm{DPW}}\eta
\]
with initial condition $\Phi_{\mathrm{DPW}}(0)=U$, where $U$ is the unitarizer
provided by Proposition~\ref{prop:monodromy-problem}.
The solution is defined on the simply connected domain
$\mathbb{C}$ minus the rays from $p_j$ to $\infty$, $j=1,\dots,4$.
Let $F$ denote the unitary part of $\Phi_{\mathrm{DPW}}$.

Recall that the immersion $f$ is given by the Sym--Bobenko formula
(see Remark~\ref{Rem:scaling2}),
\[
f(z)
=
-2i\,\frac{\partial F}{\partial\lambda}F^{-1}\Big|_{\lambda=1}
=
-2\,\frac{\partial F(e^{i\tau})}{\partial\tau}\Big|_{\tau=0}F(1)^{-1}.
\]

We first analyze the symmetries $\sigma(z)=\overline z$ and
$\sigma\delta(z)=-\overline z$.
As in \cite{HHT2}, we have
\[
\sigma^*\Phi_{\mathrm{DPW}}=\overline{\Phi_{\mathrm{DPW}}},\qquad
\sigma\delta^*\Phi_{\mathrm{DPW}}
=
D\,\overline{\Phi_{\mathrm{DPW}}}\,D^{-1},
\]
which implies
\[
\sigma^*f=-\overline f,\qquad
\sigma\delta^*f=-D\,\overline f\,D^{-1}.
\]
These correspond to reflections across the planes $y_1=0$ and $y_2=0$,
respectively.

Next, we consider the symmetry $\sigma\tau(z)=1/\overline z$ in the north sector
$\varphi<\arg(z)<\pi-\varphi$.
From the identity $\sigma\tau^*\eta=C\,\overline\eta\,C^{-1}$, we obtain
\[
\sigma\tau^*\Phi_{\mathrm{DPW}}
=
R\,C\,\overline{\Phi_{\mathrm{DPW}}}\,C^{-1}
\]
for some $R\in\Lambda SL(2,\mathbb{C})$.
Evaluating at $z=i$ yields
\[
U\mathcal Q
=
R\,C\,\overline U\,\overline{\mathcal Q}\,C^{-1}
=
R\,C\,U\,D\,\mathcal Q\,D^{-1}C^{-1}.
\]
Writing $U=\mathrm{diag}(u,u^{-1})$, we find
\[
R
=
U\mathcal Q C D\mathcal Q^{-1}D^{-1}U^{-1}C^{-1}
=
U L_1 U^{-1} D^{-1}C^{-1}
=
\begin{pmatrix}
u^2(\mathcal Q_{11}^2+\mathcal Q_{12}^2) & -i\mathfrak q\\
-i\mathfrak q & u^{-2}(\mathcal Q_{21}^2+\mathcal Q_{22}^2)
\end{pmatrix}.
\]
Since $U$ unitarizes $L_1$, it follows that $R\in\Lambda SU(2)$.
At $\lambda=1$ we have $\mathfrak q=0$, so $R(1)$ is diagonal.
Moreover, $u$, $\mathcal Q_{11}$, and $\mathcal Q_{22}$ are real, while
$\mathcal Q_{12}$ and $\mathcal Q_{21}$ are purely imaginary, implying that
$R(1)$ has real entries.
Thus $R(1)=\pm\Id$, and continuity from $t=0$ yields $R(1)=\Id$.

Consequently,
\[
\sigma\tau^*f
=
-2\,\frac{\partial R(e^{i\tau})}{\partial\tau}\Big|_{\tau=0}
-
C\,\overline f\,C^{-1}.
\]
The first term represents a translation, while the second term corresponds to
reflection across the plane $y_3=0$.
Since $\sigma\tau$ is an involution, the translation must be vertical.

Finally, we analyze the symmetry $\sigma\tau(z)=1/\overline z$ in the east sector
$-\varphi<\arg(z)<\varphi$.
In this case,
\[
\sigma\tau^*\Phi_{\mathrm{DPW}}
=
S\,C\,\overline{\Phi_{\mathrm{DPW}}}\,C^{-1}
\]
for some $S\in\Lambda SL(2,\mathbb{C})$.
Evaluating at $z=1$ gives
\[
S
=
U\mathcal P C\mathcal P^{-1}U^{-1}C^{-1}
=
U L_3 U^{-1}C^{-1}
=
\begin{pmatrix}
u^2(\mathcal P_{11}^2-\mathcal P_{12}^2) & -\mathfrak p\\
\mathfrak p & u^{-2}(\mathcal P_{22}^2-\mathcal P_{21}^2)
\end{pmatrix}.
\]
It follows that $S\in\Lambda SU(2)$ and that $S(1)$ has real entries.
By Proposition~\ref{prop:monodromy-problem}, we have
$\mathfrak p=\pm\sin(2\pi t)$; continuity from $(t,\lambda)=(0,1)$, where
$\widehat{\mathfrak p}=-2\pi$, implies
$\mathfrak p=-\sin(2\pi t)$.
Thus
\[
S(1)
=
\begin{pmatrix}
\cos(2\pi t) & \sin(2\pi t)\\
-\sin(2\pi t) & \cos(2\pi t)
\end{pmatrix}.
\]
This yields
\[
\sigma\tau^*f
=
-2\,\frac{\partial S(e^{i\tau})}{\partial\tau}\Big|_{\tau=0}S(1)^{-1}
-
S(1)\,C\,\overline f\,C^{-1}S(1)^{-1}.
\]
The first term again represents a translation, while the second corresponds to
reflection across the plane
\[
\sin(2\pi t)y_2=\cos(2\pi t)y_3.
\]
\end{proof}


\section*{Acknowledgements}{\small
The second author is supported by the Beijing Natural Science Foundation
IS23003. The third author is supported by the French ANR project Min-Max (ANR-19-CE40-0014).}

\end{document}